\documentclass[10pt]{amsart}

\usepackage{amsmath, amsfonts, amsthm, amssymb, graphicx, fullpage, enumerate}

\newtheorem{theorem}{Theorem}[section]      
\newtheorem{lemma}{Lemma}[section]

\newtheorem{proposition}{Proposition}[section]       
 
\newtheorem*{main theorem}{Main Theorem}

\theoremstyle{remark}    
  
\theoremstyle{definition}  
\newtheorem{definition}{Definition}[section]   
      
\def\cQ{\mathcal{Q}}

\def\n{\mathbf{n}}

\def\i{\textnormal{i}}
\def\T{\mathbb{T}}         
\def\N{\mathbb{N}}     
         
\def\R{\mathbb{R}}     
\def\Z{\mathbb{Z}}
\def\V{\mathcal{V}}     
\def\C{\mathbb{C}}  
\def\P{\mathcal{P}}

\def\S{\mathcal{S}}

\def\bar#1{\overline{#1}} 

\begin{document}
\title{Difference Sets and Polynomials } 
\author{Neil Lyall \quad\quad\quad Alex Rice }

\address{Department of Mathematics, The University of Georgia, Athens, GA 30602, USA}
\email{lyall@math.uga.edu}
\address{Department of Mathematics, University of Rochester, Rochester, NY 14627 USA}
\email{alex.rice@rochester.edu} 
\begin{abstract} We provide upper bounds on the largest subsets of $\{1,2,\dots,N\}$ with no differences of the form $h_1(n_1)+\cdots+h_{\ell}(n_{\ell})$ with $n_i\in \N$ or $h_1(p_1)+\cdots+h_{\ell}(p_{\ell})$ with $p_i$ prime, where $h_i\in \Z[x]$ lie in in the classes of so-called intersective and $\P$-intersective polynomials, respectively. For example, we show that a subset of $\{1,2,\dots,N\}$  free of nonzero differences of the form $n^j+m^k$ for fixed $j,k\in \N$  has density at most $e^{-(\log N)^{\mu}}$ for some $\mu=\mu(j,k)>0$. Our results, obtained by adapting two Fourier analytic, circle method-driven strategies, either recover or improve upon all previous results for a single polynomial. 
\end{abstract}
\maketitle
\setlength{\parskip}{7pt}   

\section{Introduction}  
\subsection{Background} Lov\'asz posed the following question: If $A\subseteq \N$ contains no pair of distinct elements which differ by a perfect square, must it be the case that $$\lim_{N\to \infty} \frac{|A\cap [1,N]|}{N}=0 \ ?$$ Here and throughout we use $[1,N]$ to denote $\{1,2,\dots,N\}$. Erd\H{o}s posed the analogous question with ``perfect square" replaced by ``one less than a prime number". Furstenberg \cite{Furst} answered the former question in the affirmative via ergodic theory, specifically his correspondence principle, but obtained no quantitative information on the rate at which the density must decay. Independently, S\'ark\"ozy \cite{Sark1, Sark3}  showed via Fourier analysis, specifically a density increment argument driven by the Hardy-Littlewood circle method, that if $A\subseteq [1,N]$ contains no nonzero square differences, then \begin{equation}\label{sark1bound} \frac{|A|}{N} \ll \Big(\frac{(\log \log N)^2}{\log N}\Big)^{1/3},\end{equation} while if $a-a'\neq p-1$ (or $p+1$) for all $a,a'\in A$ and all primes $p$, then \begin{equation}\label{sark3bound} \frac{|A|}{N} \ll \frac{(\log\log\log N)^3\log\log\log\log N}{(\log \log N)^2}.\end{equation} We use ``$\ll$" to denote ``less than a constant times", with subscripts indicating what parameters, if any, the implied constant depends on.

\subsection{Improvements and extensions} Using a more intricate Fourier analytic argument, Pintz, Steiger, and Szemer\'edi \cite{PSS} improved (\ref{sark1bound}) to \begin{equation}\label{PSSbound} \frac{|A|}{N} \ll (\log N)^{-c\log\log\log\log N}, \end{equation} with $c=1/12$, and by incorporating more delicate analytic number theory results into S\'ark\"ozy's original method, Ruzsa and Sanders \cite{Ruz} dramatically improved (\ref{sark3bound}) to \begin{equation*}\frac{|A|}{N} \ll e^{-c(\log N)^{1/4}}, \end{equation*} where $c>0$ is an absolute constant.

\noindent A natural generalization of Lov\'asz's question is to extend from perfect squares to the image of more general polynomials. Balog, Pelikan, Pintz, and Szemer\'edi \cite{BPPS} extended (\ref{PSSbound}) to sets with no $k^{\text{th}}$ power differences for a fixed $k\in \N$, with $c=1/4$ and the implied constant depending on $k$. 

\noindent More generally, to hope for such a result for a given  polynomial $h \in \Z[x]$, it is clearly necessary that $h$ has a root modulo $q$ for every $q \in \N$, as otherwise there is a set $q\N$ with positive density and no differences in the image of $h$. It follows from a theorem of Kamae and Mend\`es France \cite{KMF} that this condition is also sufficient, in a qualitative sense, and in this case we say that $h$ is an \textit{intersective polynomial}. Equivalently, a polynomial is intersective if it has a $p$-adic integer root for every prime $p$. Intersective polynomials include any polynomial with an integer root and any polynomial with two rational roots with coprime denominators, but there are also intersective polynomials with no rational roots, such as $(x^3-19)(x^2+x+1)$. 

\noindent It is a theorem of Lucier \cite{Lucier}, with minor improvements exhibited by the first author and Magyar \cite{LM} and the second author \cite{thesis}, that if $h\in \Z[x]$ is an intersective polynomial of degree $k\geq 2$ and $A\subseteq[1,N]$ has no nonzero differences in the image of $h$, then \begin{equation*}\frac{|A|}{N} \ll_h \Big(\frac{\log\log N}{\log N}\Big)^{1/(k-1)}. \end{equation*} Further, Hamel and the authors \cite{HLR} extended (\ref{PSSbound}) to all intersective polynomials of degree two, for any $c<1/\log(3)$ and the implied constant depending on $c$ and the polynomial. 

\noindent To hybridize the aforementioned results, one could ask for a density bound on a set free of differences of the form $h(p)$, for a fixed $h\in \Z[x]$ and $p$ prime, but this requires further restrictions on the polynomial. Specifically, for every $q\in \N$, there must exist $r\in \Z$ with $(r,q)=1$ and $q\mid h(r)$, as otherwise $h(p)$ is divisible by $q$ for only finitely many primes $p$, and hence $mq\N$ has no differences of the form $h(p)$ for sufficiently large $m$. If this condition is satisfied we say that $h$ is a $\P$\textit{-intersective polynomial}. Equivalently, a polynomial is $\P$-intersective if for every prime $p$ it has a $p$-adic integer root that does not reduce to $0$ modulo $p$. Examples include any polynomial with a root at $1$ or $-1$, and any polynomial with rational roots $a/b$ and $c/d$ satisfying $(ab,cd)=1$, while  $(x^3-19)(x^2+x+1)$ again serves as an example free of rational roots. 

\noindent The second author \cite{Rice} showed that if $h\in\Z[x]$ is $\P$\textit{-intersective} of degree $k\geq 2$ and $\epsilon>0$, then a set $A\subseteq [1,N]$ with no nonzero differences of the form $h(p)$ satisfies $$\frac{|A|}{N} \ll_{h,\epsilon} (\log N)^{-\frac{1}{2k-2}+\epsilon},$$ and further (\ref{PSSbound}) holds with $c=(2\log(3))^{-1}-\epsilon$ if $k=2$.

\noindent Here we have only alluded to the best-known results in each case, all established through versions of the two aforementioned Fourier analytic attacks. For intermediate results and alternative proofs, the reader may refer to \cite{Green}, \cite{Slip}, \cite{LM}, \cite{Lucier2}, \cite{lipan}, \cite{Lyall}, and \cite{taoblog}.

\subsection{Main results} Here we adapt the known Fourier analytic strategies to handle differences of the form $h_1(n_1)+\cdots+h_{\ell}(n_{\ell})$ for a collection of polynomials $h_1,\dots,h_{\ell} \in \Z[x]$, as well as incorporate exponential sum estimates of Shparlinski \cite{Shpar} previously unused in this context, to establish the following result.

\begin{theorem}\label{main1} Suppose $\ell_1,\ell_2,\ell_3 \geq 0$ are integers and let $\ell=\ell_1+\ell_2+\ell_3$. 

\noindent Suppose $h_1,\dots, h_{\ell_1} \in \Z[x]$ are nonzero intersective polynomials, $h_{\ell_1+1},\dots,h_{\ell_1+\ell_2} \in \Z[x]$ are nonconstant monomials, and $h_{\ell_1+\ell_2+1},\dots, h_{\ell}\in \Z[x]$ are nonmonomials with $h_i(0)=0$ for $\ell_1+\ell_2+1\leq i \leq \ell$. 

\noindent Let $$D=\Big(\sum_{i=1}^{\ell_1} k_i^{-1}+\ell_2/2+\sum_{i=\ell_1+\ell_2+1}^{\ell} r_i^{-1}\Big)^{-1},$$ where $k_i=\deg(h_i)$ and $r_i$ is the number of nonzero coefficients of $h_i$. 

\noindent Suppose further that $A\subseteq [1,N]$ and $$a-a' \neq \sum_{i=1}^{\ell} h_i(n_i)$$ for all distinct pairs $a,a'\in A$ and for all $n_1,\dots, n_{\ell} \in \N$ with $h_1(n_1),\dots, h_{\ell}(n_{\ell}) \neq 0$.

\noindent Then, \begin{equation*}\frac{|A|}{N} \ll_{h_1,\dots, h_{\ell}, \epsilon} \begin{cases}  (\log N)^{-\frac{1}{2(D-1)}+\epsilon} & \text{if } D>2 \text{ and } \ell_3>0, \text{ for any }\epsilon>0 \\\\  \Big(\frac{\log\log N}{\log N}\Big)^{1/(D-1)} & \text{if } D>2 \text{ and } \ell_3=0 \\\\  (\log N)^{-\mu\log\log\log\log N} & \text{if }  1<D\leq2 \\\\ e^{-(\log N)^{\mu}} &\text{if } D=1 \text{ and } \ell_2+\ell_3>0 \\\\ e^{-c(\log N)^{1/4}} &\text{if } D<1 \text{ and } \ell_2+\ell_3>0  \\\\ e^{-c\sqrt{\log N}} & \text{if } D\leq 1 \text{ and } \ell_2=\ell_3=0  \end{cases},
\end{equation*}
where $c=c(h_1,\dots,h_{\ell})>0$ and $\mu=\mu(k_1,\dots,k_{\ell})>0$.
\end{theorem}

\noindent  The statement of Theorem \ref{main1} is admittedly rather obnoxious, and the reader is encouraged to refer to Sections \ref{exp1} and \ref{exp2} for discussions of several digestible and illustrative special cases of the theorem, the reasoning for the specified partitioning of the collection of polynomials and its impact on the resulting bounds, and the origin and significance of the quantity $D$.  We also establish an analogous result in the prime input setting, where the statement is weaker but more straightforward. 

\begin{theorem}\label{main2} Suppose $h_1, \dots, h_{\ell} \in \Z[x]$ are $\P$-intersective polynomials with $\deg(h_i)=k_i>0$, and let $$D'=\Big(\sum_{i=1}^{\ell} k_i^{-1} \Big)^{-1}.$$ If $A\subseteq [1,N]$ and $$a-a' \neq \sum_{i=1}^{\ell} h_i(p_i)$$ for all distinct pairs $a,a' \in A$ and for all primes $p_1,\dots, p_{\ell}$ with $h_1(p_1),\dots, h_{\ell}(p_{\ell}) \neq 0$, then 
\begin{equation*}\frac{|A|}{N} \ll_{h_1,\dots, h_{\ell}, \epsilon} \begin{cases} (\log N)^{-\frac{1}{2(D'-1)}+\epsilon} & \text{if } D'>2, \text{ for any }\epsilon>0\\\\ (\log N)^{-\mu\log\log\log\log N} & \text{if } 1< D'\leq 2\\\\ e^{-(\log N)^{\mu}} & \text{if } D'=1, \  \ell>1\\\\ e^{-c(\log N)^{1/4}} & \text{if } D'< 1\text{ or } \ell=k_1=1\end{cases},
\end{equation*}  
where $c=c(h_1,\dots,h_{\ell})>0$ and $\mu=\mu(k_1,\dots,k_{\ell})>0$.
\end{theorem}
 
\noindent \textit{Remark on constants in Theorems \ref{main1} and \ref{main2}.} By utilizing the precise statements of Lemmas \ref{gauss2}, \ref{gaussS}, and \ref{gsPI}, one can take the constant $c$ in Theorems \ref{main1} and \ref{main2} to be the reciprocal of the maximum over the collection of polynomials of the right hand side of the inequality in the conclusion of Lemma \ref{content}, times a constant depending only on $k_1,\dots,k_{\ell}$.  Further, in the first appearance of the constant $\mu$ in each theorem, one can take $\mu=1/2\log(\min\{k_i\})$ in Theorem \ref{main1}, as explicitly shown in Section \ref{hardproof}, and $\mu=1/4\log(\min\{k_i\})$ in Theorem \ref{main2}. In the second appearance of $\mu$ in each theorem, one can apply Chen's \cite{Chen} explicit bounds on the implied constants in Lemma \ref{gauss2} and take $\mu=\exp(-10\sum_{i=1}^{\ell}k_i)$. We also note that at the expense of the implied constants in Theorems \ref{main1} and \ref{main2}, we are free in all cases to assume that the main parameter $N$ is sufficiently large with respect to the fixed polynomials $h_1,\dots,h_{\ell}$, so we take this as a perpetual hypothesis and refrain from explicitly including it further. 
 
\subsection{Some special cases} \label{exp1} We first note that in the case of $\ell=1$, that is to say the previously treated cases of a single polynomial, Theorems \ref{main1} and \ref{main2} simply recover the previous best-known results, with the notable exception of ``sparse" polynomials in Theorem \ref{main1}. For example, if $h(x)=x+2x^{17}+x^{31}$, then we can take $D=3$ with $\ell_3>0$, so a set $A\subseteq [1,N]$ free of nonzero differences of the form $h(n)$ satisfies $$\frac{|A|}{N} \ll_{h,\epsilon} (\log N)^{-\frac{1}{4}+\epsilon}$$ for any $\epsilon>0$, whereas previously $1/4$ was replaced with $1/30$. Moreover, if $r=2$, in other words $h(x)=ax^j+bx^k$ for some natural numbers $j< k$ and nonzero $a,b \in \Z$, then (\ref{PSSbound}) holds with $c=1/2\log(k)$, a bound previously only known for monomials and quadratics.

\noindent In certain cases with $\ell=2$, Theorem \ref{main1} provides density bounds superior to any attained in the single polynomial case. For example, if a set $A\subseteq[1,N]$ lacks nonzero differences of the form $h_1(m)+h_2(n)$, where $h_1,h_2\in\Z[x]$ are arbitrary intersective quadratic polynomials, including for example sums of two squares, then we can take $D=1$ with $\ell_2>0$ and $\ell_3=0$, and therefore \begin{equation} \label{Bbound} \frac{|A|}{N} \ll_{h_1,h_2} e^{-c\sqrt{\log N}} \end{equation} for some $c=c(h_1,h_2)>0$. This leap in bound quality is not especially surprising in this particular case, as the collection of elements of the form $h_1(m)+h_2(n)$ inside of $[1,N]$ has size $N^{1-o(1)}$ (much like the $p-1$ case), and is hence far denser than the image of any single nonlinear polynomial. 
 
\noindent What is more notable, perhaps, is that we exhibit bounds of similar quality in certain cases where the collection of avoided differences is quite sparse. For example, if $A\subseteq[1,N]$ lacks differences of the form $m^j+n^k$ for fixed $j,k\in \N$, then $D=1$, $\ell_2>0$, $\ell_3=0$, and \begin{equation} \label{cbound} \frac{|A|}{N} \ll_{j,k} e^{-c(\log N)^{\mu}} \end{equation} for some $c=c(j,k),\mu=\mu(j,k)>0$, despite the fact that the collection of elements of the form $m^j+n^k$ in $[1,N]$ has size at most $N^{\frac{1}{j}+\frac{1}{k}}$. In fact, the same estimate holds for a set $A\subseteq [1,N]$ free of differences of the form $h_1(m)+h_2(n)$ where  $h_1(x)=a_1x^{j_1}+b_1x^{k_1}$ and $h_2(x)=a_2x^{j_2}+b_2x^{k_2}$ for nonzero $a_i,b_i \in \Z$ and natural numbers $j_i<k_i$, and $c=c(h_1,h_2)>0$. 

\noindent As discussed in Section \ref{exp2}, the application of the exponential sum estimates utilized to achieve this gain for sparse polynomials and monomials requires sieve estimates that impose a limitation on the density increment iteration. This limitation results in the factor of two loss with $\ell_2+\ell_3>0$ in the exponents of certain bounds in Theorem \ref{main1}. For example, if $\ell=1$ and $h(x)=x+x^3+x^4$, we are better off treating $h$ as an arbitrary intersective polynomial of degree $4$, resulting in a density bound of about $(\log N)^{-1/3}$, as opposed to a polynomial with three nonzero coefficients, yielding a bound of about $(\log N)^{-1/4}.$ This ambiguity in optimal partition where a larger $D$ can yield a better bound is reasonably rare, and only occurs when all possible values of $D$ are greater than $2$ or all possible values of $D$ are at most $1$. 

\noindent For some notable examples with $\ell=2$ in Theorem \ref{main2}, we see that if $A\subseteq [1,N]$ has no differences of the form $(p-1)^2+(q-1)^2$ with $p,q$ prime, then $D=1$ and (\ref{cbound}) holds with all constants absolute, whereas if $A$ lacks differences of the form $(p-1)^4+(q-1)^4$ with $p,q$ prime, then $D=2$ and (\ref{PSSbound}) holds with $c=1/4\log(4)$, and of course these powers can be replaced with any pairs of $\P$-intersective polynomials of degree $2$ or $4$, respectively.

\noindent We also note that the specially earmarked case of $\ell=k_1=1$ in Theorem \ref{main2} is simply the previously studied case of $h(p)=a(p\pm 1)$, and we include the necessary tools to recover this result for the sake of completeness.

\subsection{Motivation for $D$ and $D'$ from Gauss sum estimates} \label{exp2} In this context, the guiding principle of the Hardy-Littlewood circle method is that if $h\in \Z[x]$, then the \textit{Weyl sum} \begin{equation} \label{weyls} \sum_{n=1}^Me^{2\pi \textnormal{i}h(n)\alpha}\end{equation} is much smaller than the trivial bound $M$, unless $\alpha$ is well-approximated by a rational number with small denominator. 

\noindent On a coarse scale, this principle is captured by combining the pigeonhole principle with Weyl's Inequality (see Lemma \ref{weyl2}), but for a more refined treatment, we must address the following question: If $\alpha$ IS quite close to a rational with a quite small denominator, for example smaller than a tiny power of $M$, can we beat the trivial bound at all? 

\noindent This question turns out to be quite straightforward, as under these conditions (\ref{weyls}) has a convenient asymptotic formula, and the gain from the trivial bound resides in a local version of the sum, or \textit{Gauss sum}. Specifically, if $\alpha$ is close to $a/q$ with  $q$ small, then, up to a small error, the magnitude of (\ref{weyls}) is at most $M$ times  \begin{equation}\label{gsintro} q^{-1}\sum_{s=0}^{q-1}e^{2\pi \text{i}h(s)a/q}. \end{equation} 

\noindent Moreover, by sieving our initial set of inputs, letting $W$ equal a product of small primes and considering only inputs coprime to $W$, we can replace (\ref{weyls}) with \begin{equation} \label{weyls2} \sum_{\substack{n=1 \\ (n,W)=1}}^Me^{2\pi \textnormal{i}h(n)\alpha}, \end{equation} in which case the gain from the trivial bound for $\alpha$ near $a/q$ with small $q$ is given roughly by \begin{equation}\label{gsintro2} q^{-1}\sum_{\substack{s=0 \\ (s,q)=1}}^{q-1}e^{2\pi \text{i}h(s)a/q}. \end{equation} 

\noindent In applying the two previously developed Fourier analytic arguments, the resulting density bounds are determined by the power $\theta=\theta(h)$ such that the magnitude of the relevant Weyl sum for $\alpha$ near $a/q$ with $q$ small beats the trivial bound by a factor of $q^{-\theta}$. In particular, to run the more intricate method developed first in \cite{PSS}, $\theta$ must be at least $1/2$.  

\noindent Moreover, when considering sums of polynomials, the relevant sum splits, for example $$\sum_{n_1}^{M_1}\sum_{n_2}^{M_2}e^{2\pi \text{i} (h_1(n_1)+h_2(n_2))\alpha}=\Big(\sum_{n=1}^{M_1}e^{2\pi \textnormal{i}h_1(n)\alpha}\Big)\Big(\sum_{n=1}^{M_2}e^{2\pi \textnormal{i}h_2(n)\alpha}\Big),$$ so we can add together the corresponding powers $\theta(h_1)$ and $\theta(h_2)$. 

\noindent In the context of Theorem \ref{main1}, we are free, for a given polynomial $h\in \Z[x]$, to choose the better of (\ref{gsintro}) and (\ref{gsintro2}), with the caveat that when employing the more straightforward of the two methods, choosing (\ref{gsintro2}) imposes an increased limitation on the density increment iteration due to the need to accurately count integers with no small prime factors. 

\noindent Traditional estimates (see Lemma \ref{gauss2}) say that for any $h\in \Z[x]$, one can choose (\ref{gsintro}) and take $\theta(h)=1/\deg(h)$. As observed in \cite{BPPS} (see Lemma \ref{gaussB}), if $h\in\Z[x]$ is a nonconstant monomial then one can choose (\ref{gsintro2}) and take $\theta(h)=1/2$. Finally, if $h\in\Z[x]$ has $r\geq2$ nonzero coefficients, then by estimates of Shparlinski (see Lemma \ref{gaussS}), one can choose (\ref{gsintro2}) and take $\theta(h)=1/r$. 

\noindent Given a collection of intersective polynomials $h_1,\dots,h_{\ell}\in \Z[x]$,  we choose $\theta(h_i)$ for $1\leq i \leq \ell$ and for $\alpha$ near $a/q$ with $q$ small, we can beat the trivial bound on the chosen  $\ell$-fold Weyl sum by $q^{-\theta}$ where $\theta=\sum_{i=1}^{\ell} \theta(h_i)$. 

\noindent As previously mentioned, it is this quantity $\theta$ that is the primary determining factor in the eventual density bound, with ``breaking points" at $\theta=1/2$, where the more intricate argument kicks in, and $\theta=1$, where the more straightforward argument yields particularly good bounds. The quantity $D$ defined in Theorem \ref{main1} is simply $1/\theta$, where the reciprocal is taken for aesthetic purposes, and so that $D$ plays the role formerly played by the degree of a single polynomial.

\noindent In the prime input setting of Theorem \ref{main2} the aforementioned sieve technique does not yield improved Gauss sum estimates, so we must stick to traditional gains and set $\theta(h_i)=1/\deg(h_i)$ for $1\leq i \leq \ell$, which explains the more straightforwardly defined quantity $D'$ in that theorem.

\noindent \textit{Remark on the Generalized Riemann Hypothesis.} As previously remarked, the limitations of known sieve estimates, as well as our limited knowledge of the distribution of primes in arithmetic progressions, result in potentially avoidable losses in Theorems \ref{main1} and \ref{main2}, respectively. Specifically, if we assume the Generalized Riemann Hypothesis, then the factor of $2$ can be dropped from the exponents $1/2(D-1)$ and $1/2(D'-1)$, and both appearances of the exponent $1/4$ can be changed to $1/2$.
\subsection{Lower Bounds and Conjectures} Armed with a collection of $7$ elements of $\Z/65\Z$, no distinct pair of which differ by a mod $65$ square, Ruzsa \cite{Ruz2} was able to construct a set $A\subseteq [1,N]$ with no nonzero square differences satisfying $|A|\gg N^c$, where $c=(1+\ln7/\ln65)/2 \approx 0.7331$. Recently, Lewko \cite{Lewko} made the slight improvement to $c=(1+\ln12/\ln205)/2 \approx 0.7334$. 

\noindent The finite field analog of the square difference question suggests that $c=3/4$ may be a natural limitation to Ruzsa's construction, which could potentially be viewed as evidence toward $N^{3/4}$ as the true threshold for this problem, while many believe the threshold actually grows faster than $N^{1-\epsilon}$ for any $\epsilon>0$.

\noindent For the $p-1$ case, the gap between known upper and lower bounds is even more cavernous. Ruzsa \cite{Ruz3} constructed a set $A\subset [1,N]$ satisfying $|A|\gg N^{c/\log\log N}$ with no $p-1$ differences, but nothing better in this direction is known. Consequently, the full resolutions of even the two original questions, much less the various generalizations, are still massively open.

\noindent \textit{Remark on generality of Theorems \ref{main1} and \ref{main2}.} We note that the necessary intersective condition makes perfect sense in a multivariable setting, and analogous results should hold for every intersective integral polynomial in several variables, not just diagonal forms. Further, there do exist intersective diagonal forms not covered in these theorems. For example, if $p$ is a prime congruent to $1$ modulo $90090$ that is not the sum of two integer cubes (of which there are plenty), then, since $p$ is a sum of two cubes modulo $q$ for every $q\in \N$, $x^3+y^3-p$ is an intersective polynomial in two variables that cannot be expressed as the sum of two single-variable intersective polynomials.

\noindent \textbf{Acknowledgements and Funding:} The authors would like to thank Paul Pollack and Steve Gonek for their helpful comments and references. The first author was partially supported by Simons Foundation Collaboration Grant for Mathematicians 245792.
\section{Auxiliary Polynomials and Inheritance Propositions}\label{auxsec} At some point in the proofs of all cases of Theorems \ref{main1} and \ref{main2}, we apply a density increment strategy, and we need to keep track of the inherited lack of arithmetic structure at each step of the iteration. Specifically, if we start with a set free of differences of the form $h_1(n_1)+\cdots+h_{\ell}(n_{\ell})$ for polynomials $h_1,\dots, h_{\ell}$, it spawns denser sets free of differences that are the sum of elements in new polynomial images. The following definitions describe all of the polynomials that we could potentially encounter.  

\noindent \textit{Remark on notation.} In an effort to maintain a bearable aesthetic, we frequently utilize both subscripts and superscripts for indexing purposes. Through context and consistency, we hope to avoid any confusion in distinguishing between superscript indices and exponents. 
 
\subsection{Auxiliary Polynomials}\label{auxdef} Suppose $h_1,\dots,h_{\ell} \in \Z[x]$ is a collection of intersective polynomials. For each $1\leq i \leq \ell$ and each prime $p$, we fix $p$-adic integers $z^p_i$ with $h_i(z^p_i)=0$, requiring that $z^p_i \not\equiv 0 \text{ mod }p$ if $h_i$ is $\P$-intersective. If considering the unrestricted input case of Theorem \ref{main1} and $h_i(0)=0$, we take $z^p_i=0$ for all $p$, and similarly in the prime input case of Theorem \ref{main2}, we take $z^p_i=\pm 1$ for all $p$ in the event that $h_i(\pm 1)=0$. 

\noindent \textit{Remark.} The definitions which follow certainly depend on the choice of $p$-adic integer roots, but any choice subject to the aforementioned restrictions works equally well for our purposes, and we suppress the dependence on this choice in the coming notation. 

\noindent By reducing modulo prime powers and applying the Chinese Remainder Theorem, the choices of $z^p_i$  determine, for each natural number $d$, a unique integer $r^d_i \in (-d,0]$, which consequently satisfies $d \mid h_i(r^d_i)$, and in the case that $h_i$ is $\P$-intersective we have $(r^d_i,d)=1$. 

\noindent We define the function $\lambda_i$ on $\N$ by  letting $\lambda_i(p)=p^{m_i}$ for each prime $p$, where $m_i$ is the multiplicity of $z^p_i$ as a root of $h_i$, and then extending it to be completely multiplicative. Further, we define $$\lambda=\lambda_1\circ \cdots \circ \lambda_{\ell} \quad \text{and} \quad \tilde{\lambda}_i=\lambda_1\circ\cdots \circ\lambda_{i-1}\circ\lambda_{i+1}\circ \cdots \circ \lambda_{\ell}.$$ 
\noindent For each $d\in \N$ and $1\leq i\leq\ell$, we define the \textit{auxiliary polynomial} $h^d_i$ by 
\begin{equation*} h^d_i(x)=h^d_i(r^d_i + dx)/\lambda_i(d). \end{equation*}

\noindent If $p^j \mid d$ for $p$ prime and $j \in \N$, then since $r^d_i \equiv z^p_i$ mod $p^j$, we see by factoring $h_i$ over $\Z_p$ that all the coefficients of $h_i(r^d_i+dx)$ are divisible by $p^{jm_i}$, hence each auxiliary polynomial has integer coefficients.  We also note that if $h_i(0)=0$, then the number of nonzero coefficients of $h_i^d$ is the same for all $d$.

\subsection{Inheritance Propositions} It is important to note that the leading coefficients of the auxiliary polynomials grow at least as quickly, up to a constant depending only on $h_i$, as the other coefficients. In particular, if $b^d_i$ is the leading coefficient of $h^d_i$, then for any $x>0$ we have that if $b^d_i>0$, then
\begin{equation} \label{symBIG}
\Big|\Big\{n\in \N: 0<h^d_i(n)<x\Big\} \ \triangle \ [1,(x/b^d_i)^{1/k}]\Big| \ll_{h_i} 1,
\end{equation} where $\triangle$ denotes the symmetric difference and the analogous observation holds if $b^d_i<0$. 

\noindent We define these auxiliary polynomials to keep track of the inherited lack of arithmetic structure at each step of a density increment iteration. For the unrestricted input setting in Theorem \ref{main1}, we define $$I(h)=\begin{cases} \{h(n)>0 : n\in \N\} & \text{ if }h \text{ has positive leading coefficient} \\\\ \{h(n)<0 : n\in \N\} & \text{ if }h \text{ has negative leading coefficient}  \end{cases}$$ for a nonzero polynomial $h\in \Z[x]$. 

\noindent For the prime input setting in Theorem \ref{main2}, given a collection of intersective polynomials $h_1,\dots,h_{\ell}\in \Z[x]$, we let $$\Lambda^d_{i}  =\{x \in \N: r^i_d+dx \text{ is prime}\}$$ 

\noindent for each $d\in \N$, and for a nonzero polynomial $h\in \Z[x]$ we define $$\V^d_i(h)=\begin{cases} \{h(n)>0 : n\in \Lambda^d_i\} & \text{ if }h \text{ has positive leading coefficient} \\\\ \{h(n)<0 : n\in \Lambda^d_i\} & \text{ if }h \text{ has negative leading coefficient}  \end{cases}.$$ Note that the polynomials do not need to be $\P$-intersective for these latter definitions to make sense, but if they are not then some of the sets $\Lambda^d_{i}$ are nearly if not completely empty. 

\noindent For any sets $A,B\subseteq \Z$, we use the standard notation $A\pm B=\{a \pm b : a\in A, \ b\in B\}$ for the sum and difference sets, respectively. The following two propositions make precise the aforementioned inherited lack of structure in each case.

\begin{proposition} \label{inh} If $h_1,\dots,h_{\ell}\in \Z[x]$ is a collection of intersective polynomials, $A \subseteq \N$, $$(A-A)\cap \Big(I(h^{d_1}_1)+\cdots+I(h^{d_{\ell}}_{\ell})\Big)\subseteq \{0\},$$ and $A'\subseteq \{\ell\in \N : x+\lambda(q)\ell \in A\}$, then $$(A'-A')\cap \Big(I(h^{\tilde{\lambda}_1(q)d_1}_1)+\cdots+ I(h^{\tilde{\lambda}_{\ell}(q)d_{\ell}}_{\ell})\Big)\subseteq \{0\}.$$
\end{proposition}

\begin{proof} Suppose that $A\subseteq \N$, $A'\subseteq \{\ell\in \N : x+\lambda(q)\ell \in A\}$, and 
\begin{align*}0\neq a-a'&=\sum_{i=1}^{\ell}h^{\tilde{\lambda}_i(q)d_i}_i(n_i)=\sum_{i=1}^{\ell}\frac{h_i\Big(r^{\tilde{\lambda}_i(q)d_i}_i+\tilde{\lambda}_i(q)d_in_i\Big)}{\lambda_i\Big(\tilde{\lambda}_i(q)d_i\Big)}=\sum_{i=1}^{\ell}\frac{h_i\Big(r^{\tilde{\lambda}_i(q)d_i}_i+\tilde{\lambda}_i(q)d_in_i\Big)}{\lambda(q)\lambda_i(d_i)}\end{align*} 
for some $n_1,\dots, n_{\ell}\in \N$, $a,a' \in A'$, with all polynomial terms having the same sign as the corresponding leading coefficient. By construction we know that $r^{\tilde{\lambda}_i(q)d_i}_i\equiv r^{d_i}_i$ mod $d_i$, so there exists $s_i\in \Z$ such that $r^{\tilde{\lambda}_i(q)d_i}_i=r^{d_i}_i+d_is_i$, and therefore \begin{align*}0\neq \sum_{i=1}^{\ell} h^{d_i}_i(s_i+\tilde{\lambda}_i(q)n_i)=\sum_{i=1}^{\ell}\frac{h_i(r_i^{d_i}+d_i(s_i+\tilde{\lambda}_i(q)n_i))}{\lambda_i(d_i)}=\lambda(q)(a-a').\end{align*} Because $A'\subseteq \{\ell\in \N : x+\lambda(q)\ell \in A\}$, we know that $\lambda(q)(a-a')\in A-A$, hence  $$(A-A)\cap \Big(I(h^{d_1}_1)+\cdots+I(h^{d_{\ell}}_{\ell})\Big) \not\subseteq \{0\},$$ and the contrapositive is established.
\end{proof}

\noindent We utilize the following analog of Proposition \ref{inh} in the proof of Theorem \ref{main2}.

\begin{proposition} \label{inh2} If $h_1,\dots,h_{\ell}\in \Z[x]$ is a collection of intersective polynomials, $A \subseteq \N$, $$(A-A)\cap \Big(\V_1^{d_1}(h^{d_1}_1)+\cdots+\V^{d_{\ell}}_{\ell}(h^{d_{\ell}}_{\ell})\Big)\subseteq \{0\},$$ and $A'\subseteq \{\ell\in \N : x+\lambda(q)\ell \in A\}$, then $$(A'-A')\cap \Big(\V^{\tilde{\lambda}_1(q)d_1}_1(h^{\tilde{\lambda}_1(q)d_1}_1)+\cdots+ \V^{\tilde{\lambda}_{\ell}(q)d_{\ell}}_{\ell}(h^{\tilde{\lambda}_{\ell}(q)d_{\ell}}_{\ell})\Big)\subseteq \{0\}.$$

\end{proposition}

\begin{proof}The proof is identical to that of Proposition \ref{inh}, with the added observation that if $n\in \Lambda^{\tilde{\lambda}_i(q)d}_i$ and $r^{\tilde{\lambda}_i(q)d}_i=r^{d}_i+ds$, then $s+\tilde{\lambda}_i(q)n \in \Lambda^d_i$.
\end{proof}

\section{Preliminaries  for S\'ark\"ozy's Method} \label{prelim}

In sections \ref{unr} and \ref{primesec}, we apply adapted, streamlined versions of S\'ark\"ozy's \cite{Sark1, Sark3} original $L^2$ density increment method, more closely modeled after \cite{LM}, \cite{Lucier}, and \cite{Rice}. Here we set the stage with some requisite tools and notation.

\subsection{Fourier analysis on $\Z$} We embed our finite sets in $\Z$, on which we utilize the discrete Fourier transform. Specifically, for a function $F: \Z \to \C$ with finite support, we define $\widehat{F}: \T \to \C$, where $\T$ denotes the circle  parameterized by the interval $[0,1]$ with $0$ and $1$ identified, by \begin{equation*} \widehat{F}(\alpha) = \sum_{x \in \Z} F(x)e^{-2 \pi \text{i}x\alpha}. \end{equation*}

\noindent Given $N\in \N$ and a set $A\subseteq [1,N]$ with $|A|=\delta N$, we examine the Fourier analytic behavior of $A$ by considering the \textit{balanced function}, $f_A$, defined by
\begin{equation*} f_A=1_A-\delta 1_{[1,N]}.\end{equation*}

\subsection{The circle method} We analyze the behavior of $\widehat{f_A}$ using the Hardy-Littlewood circle method, decomposing the frequency space into two pieces: the points on the circle that are close to rationals with small denominator, and those that are not.

\begin{definition}Given $\gamma>0$ and $Q\geq 1$, we define, for each $q\in \N$ and $a\in [1,q]$,
$$\mathbf{M}_{a/q}(\gamma)=\left\{ \alpha \in \T : \Big|\alpha-\frac{a}{q}\Big| < \gamma \right\},$$  $$\mathbf{M}_q(\gamma)=\bigcup_{(a,q)=1} \mathbf{M}_{a/q}(\gamma),$$ and $$ \mathbf{M}'_q(\gamma)=\bigcup_{r\mid q} \mathbf{M}_r(\gamma)=\bigcup_{a=1}^q \mathbf{M}_{a/q}(\gamma).$$
We then define $\mathfrak{M}(\gamma,Q)$, the \textit{major arcs}, by
\begin{equation*} \mathfrak{M}(\gamma,Q)=\bigcup_{q=1}^{Q} \mathbf{M}_q(\gamma),
\end{equation*} 
and $\mathfrak{m}(\gamma,Q)$, the \textit{minor arcs}, by  
\begin{equation*} \mathfrak{m}(\gamma,Q)=\T\setminus \mathfrak{M}(\gamma,Q).
\end{equation*} 
We note that if $2\gamma Q^2<1$, then \begin{equation} \label{majdisj}\mathbf{M}_{a/q}(\gamma)\cap\mathbf{M}_{b/r}(\gamma)=\emptyset \end{equation}whenever  $a/q\neq b/r$ and  $q,r \leq Q$. 
\end{definition}  

\subsection{Preliminary notation} Before delving into the details of the arguments for Theorems \ref{main1} and \ref{main2}, we take the opportunity to define some relevant sets and quantities, depending on polynomials $h_1,\dots,h_{\ell}\in \Z[x]$, a partition $\ell=\ell_1+\ell_2+\ell_3$, scaling parameters $d_1,\dots,d_{\ell}$, a parameter $\eta>0$, and the size of the ambient interval $N$, that will be used in both cases. In all the notation defined below, we suppress all of the aforementioned dependence, as the relevant objects will be fixed in context. 

\noindent To this end, given intersective polynomials $h_1,\dots,h_{\ell}\in \Z[x]$, we let $k_i=\deg(h_i)$, $\textbf{k}=(k_1,\dots,k_{\ell})$, $k=\prod_{i=1}^{\ell} k_i$, $K=2^{10k}$, and $$D'=\Big(\sum_{i=1}^{\ell} k_i^{-1}\Big)^{-1}.$$ For ease of notation when specifying the dependence of constants on all of these polynomials, we use $\textnormal{\textbf{h}}$ to denote $(h_1,\dots,h_{\ell})$. Further, when working in $[1,N]$ with scaling parameters $d_1,\dots,d_{\ell}\in \N$, we define the following for $1\leq i \leq \ell$:
$$M_i=\lfloor (N/9\ell|b_{i}|)^{1/k_i}\rfloor,$$ where $b_i$ is the leading coefficient of $h_i^{d_i}$, and $$H_{i}=\begin{cases}\{n \in \N : 0<h^{d_i}_i(n)<N/9\ell\} & \text{if } b_i>0 \\\{n \in \N : -N/9\ell<h^{d_i}_i(n)<0\} & \text{if } b_i<0\end{cases},$$ noting that by (\ref{symBIG}) we have \begin{equation}\label{symdif} |H_{i} \ \triangle  \ [1,M_i]| \ll_{h_i} 1.  \end{equation}
For $\n\in \N^{\ell}$, we let $h(\n)=h^{d_1}_1(n_1)+\cdots+h^{d_{\ell}}_{\ell}(n_\ell)$, and we let $Z=\{\n \in \N^{\ell} : n_i\in H_i, \  h(\n)=0\}$.

\noindent In the context of Theorem \ref{main1} in Section \ref{unr}, given nonnegative integers $\ell_1,\ell_2,\ell_3$ with $\ell_1+\ell_2+\ell_3=\ell$, we fix a real number $\epsilon$ by setting $\epsilon=0$ if $\ell_2=\ell_3=0$ and letting $\epsilon>0$ be an arbitrary positive number if $\ell_2+\ell_3>0$. If navigating the argument with a particular collection of polynomials in mind and $\ell_2=\ell_3=0$, one can replace $\epsilon$ with $0$ throughout and any dependence of constants on this parameter should be ignored. 

\noindent Also in Section \ref{unr}, we employ a trick of initially sieving our input values in order to improve exponential sum estimates. To this end, given $\eta>0$, we let $$\displaystyle{W=\prod_{p\leq \eta^{-(D'+\epsilon)}}p},$$ where the product is taken over primes. For $1\leq i \leq \ell$ we let $\tilde{H}_i=\{n\in H_i : (n,W)=1\}$ and $$\displaystyle{\tilde{M}_i=M_i\prod_{p\leq \eta^{-(D'+\epsilon)}} \Big(1-\frac{1}{p}\Big)}.$$ Further, we define $\tilde{H}=\Big(H_1\times \cdots \times H_{\ell_1}\times \tilde{H}_{\ell_1+1} \cdots \times \tilde{H}_{\ell}\Big) \setminus Z$   and $$\tilde{M}=\prod_{i=1}^{\ell_1}M_i\prod_{j=\ell_1+1}^{\ell} \tilde{M_j}.$$ By (\ref{symdif}) we see that $|Z|\ll_{\textbf{k}} \prod_{i=1}^{\ell-1}M_i$, and further noting the standard estimate 
\begin{equation}\label{loglog} \prod_{p\leq X}\Big( 1-\frac{1}{p}\Big) \gg (\log X)^{-1},\end{equation} we see that in fact \begin{equation}\label{Hbig}  |\tilde{H}| \geq \tilde{M}/2, \end{equation} provided, for example, that $\eta^{-(D'+\epsilon)}<N^{1/10}$.

\noindent In the context of Theorem \ref{main2} in Section \ref{primesec}, the aforementioned sieving does not yield the desired gains, so we make the more straightforward definitions  $H=\Big(H_1\times \cdots \times H_{\ell}\Big) \setminus Z$ and $M=\prod_{i=1}^{\ell}M_i,$ analogously noting that \begin{equation} \label{Hbig2} |H|\geq M/2. \end{equation}

\section{S\'ark\"ozy's Method: Theorem \ref{main1} for $D>2$ and $D\leq 1$} \label{unr}
For the remainder of this section, we fix intersective polynomials $h_1,\dots, h_{\ell} \in \Z[x]$, partitioned into groups of size $\ell_1, \ell_2,\ell_3\geq 0$ as indicated in Theorem \ref{main1}. Namely, $h_1,\dots,h_{\ell_1}$ are arbitrary nonzero intersective polynomials, $h_{\ell_1+1},\dots, h_{\ell_1+\ell_2}$ are nonconstant monomials, and $h_{\ell_1+\ell_2+1}, \dots, h_{\ell}$ are nonmonomials with no constant term. Throughout the argument, when working in $[1,N]$, we let $c_1=(10kK)^{-1}$ and define $$\cQ=\cQ(N)=\begin{cases}N^{c_1} &\text{if } \ell_2=\ell_3=0 \\ e^{c_1\sqrt{\log N}} &\text{if } \ell_2+\ell_3>0 \end{cases}.$$

\noindent We deduce Theorem \ref{main1} (outside of $1<D\leq 2$) from the following iteration lemma, which states that a set  deficient in the desired arithmetic structure spawns a new, significantly denser subset of a slightly smaller interval with an inherited deficiency in the structure associated to appropriate auxiliary polynomials. 

\begin{lemma} \label{mainit} Suppose $A\subseteq [1,N]$ with $|A|=\delta N$. If $(A-A)\cap \Big(I(h^{d_1}_1)+\cdots+ I(h^{d_{\ell}}_{\ell})\Big)\subseteq \{0\}$ and $d_1,\dots,d_{\ell},\delta^{-1}\leq \cQ$, then there exist $q\ll_{\textnormal{\textbf{h}},\epsilon} \delta^{-(D+\epsilon)}$ and $A'\subseteq [1,N']$ with 
$N'\gg_{\textnormal{\textbf{h}},\epsilon}  \delta^{(D'+\epsilon)(k+1)}N$, 

\begin{equation*}\frac{|A'|}{N'} \geq \begin{cases} \delta+c\delta^{D+2\epsilon}& \text{if } D>1  \\ (1+c\log^{-C}(\delta^{-1}))\delta  & \text{if } D=1, \ell_2+\ell_3>0 \\ (1+c)\delta & \text{if } D= 1, \ell_2+\ell_3=0 \text{ or } D<1\end{cases} , \end{equation*} and $$(A'-A')\cap \Big(I(h^{\tilde{\lambda}_1(q)d_1}_1)+\cdots+ I(h^{\tilde{\lambda}_{\ell}(q)d_{\ell}}_{\ell})\Big)\subseteq \{0\}, $$  for some $c=c(\textnormal{\textbf{h}},\epsilon)>0$ and $C=C(\textnormal{\textbf{k}})$. 
\end{lemma}

\subsection*{Proof that Lemma \ref{mainit} implies Theorem \ref{main1} for $D>2$ and $D \leq 1$} Throughout this proof, we let $C$ and $c$ denote sufficiently large or small positive constants, respectively, which we allow to change from line to line, but can depend only on $\textbf{h}$ and $\epsilon$. We use $C'$ and $c'$ similarly, but these constants can depend only on $\textbf{k}$. Suppose $A \subseteq [1,N]$ with $|A|=\delta N$ and $$(A-A)\cap \Big(I(h_1)+\cdots+I(h_{\ell})\Big)\subseteq \{0\}.$$ Setting $A_0=A$, $N_0=N$, $d^0_1,\dots,d^0_{\ell}=1$, and $\delta_0=\delta$, Lemma \ref{mainit} yields, for each $m$, a set $A_m \subseteq [1,N_m]$ with $|A_m|=\delta_mN_m$ and $$(A-A)\cap \Big(I(h^{d^m_1}_1)+\cdots+I(h^{d^m_{\ell}}_{\ell})\Big)\subseteq \{0\}$$ satisfying
\begin{equation} \label{NmI} N_m \geq c\delta^{(D'+\epsilon)(k+1)}N_{m-1} \geq (c\delta)^{(D'+\epsilon)(k+1)m} N,
\end{equation}
 
\begin{equation} \label{incsizeI} \delta_m \geq \begin{cases} \delta_{m-1}+c\delta_{m-1}^{D+2\epsilon}& \text{if } D>1  \\ (1+c\log^{-C'}(\delta_{m-1}^{-1}))\delta_{m-1}  & \text{if }D=1, \ell_2+\ell_3>0\\ (1+c)\delta_{m-1} & \text{if } D= 1, \ell_2+\ell_3=0 \text{ or } D<1\end{cases}
\end{equation}
and 
\begin{equation}\label{dmI} d^m_i \leq (c\delta)^{-k(D+\epsilon)}d^{m-1}_i \leq (c\delta)^{-k(D+\epsilon)m},
\end{equation}
as long as 
\begin{equation} \label{delmI} d^m_i, \delta_m^{-1} \leq \begin{cases}N_m^{c_1} &\text{if } \ell_2=\ell_3=0 \\ e^{c_1\sqrt{\log N_m}} &\text{if } \ell_2+\ell_3>0 \end{cases}.
\end{equation}
If $D>1$ and $\ell_2=\ell_3=0$ (and hence $\epsilon=0$), then by (\ref{incsizeI}) we see that the density $\delta_m$ will surpass $1$, and hence (\ref{delmI}) must fail, for $m=C\delta^{-(D-1)}$. In particular, by (\ref{NmI}) and (\ref{dmI}) we must have 
$(c\delta)^{-C\delta^{-(D-1)}}\geq N,$
which implies
\begin{equation*} \delta \ll_{\textnormal{\textbf{h}}} \Big(\frac{\log\log N}{\log N}\Big)^{1/(D-1)},
\end{equation*}
as required. 

\noindent If $D>1$ and $\ell_2+\ell_3>0$, then we make the same observation for $m=C\delta^{-(D-1+2\epsilon)}$, and hence by (\ref{NmI}) and (\ref{dmI}) we must have 
$(c\delta)^{-C\delta^{-(D-1+2\epsilon)}}\geq e^{\sqrt{\log N}},$
which implies
\begin{equation*} \delta \ll_{\textnormal{\textbf{h}},\epsilon} \Big(\frac{\log\log N}{\log N}\Big)^{1/2(D-1+2\epsilon)} \ll_{\textnormal{\textbf{h}},\epsilon} (\log N)^{-1/2(D-1+3\epsilon)}.
\end{equation*}

\noindent Further, if $D=1$ and $\ell_2+\ell_3>0$, then (\ref{delmI}) must fail for $m=C\log^{C'}(\delta^{-1})$, which by (\ref{NmI}) and (\ref{dmI}) yields $(c\delta)^{-C\log^{C'}(\delta^{-1})}\geq e^{\sqrt{\log N}},$ and hence $$\delta \ll_{\textnormal{\textbf{h}}} e^{-(\log N)^{c'}}.$$

\noindent If $D= 1$, $\ell_2+\ell_3=0$, then  we see that  (\ref{delmI}) must fail for $m=C\log(\delta^{-1})$, and by (\ref{NmI}) and (\ref{dmI}) we must have 
$(c\delta)^{-C\log(\delta^{-1})}\geq e^{\sqrt{\log N}},$  which implies $$\delta \ll_{\textnormal{\textbf{h}}} e^{-c(\log N)^{1/4}}.$$ 

\noindent Finally, if $D\leq 1$ and $\ell_2=\ell_3=0$, then again (\ref{delmI}) must fail for $m=C\log(\delta^{-1})$, so by (\ref{NmI}) and (\ref{dmI}) we must have 
$(c\delta)^{-C\log(\delta^{-1})}\geq N,$ and therefore $$\delta \ll_{\textnormal{\textbf{h}}} e^{-c\sqrt{\log N}}.$$ We have now established nontrivial bounds in all cases, and these bounds match the claims in Theorem \ref{main1} outside of the range $1<D\leq2$. \qed

\noindent The philosophy behind the proof of Lemma \ref{mainit} is that a deficiency in the desired arithmetic structure from a set $A$ represents  nonrandom behavior, which should be detected in the Fourier analytic behavior of $A$. Specifically,  we locate one small denominator $q$ such that $\widehat{f_A}$ has $L^2$ concentration around rationals with denominator $q$, then use that information to find a long arithmetic progression on which $A$ has increased density.

\begin{lemma}  \label{L2I} Suppose $A\subseteq [1,N]$ with $|A|=\delta N$, let $\eta=c_0\delta$ for a sufficiently small constant $c_0=c_0(\textnormal{\textbf{h}},\epsilon)>0$, and let $\gamma=\eta^{-(D'+\epsilon)}/N$. If $(A-A)\cap \Big(I(h^{d_1}_1)+\cdots+ I(h^{d_{\ell}}_{\ell})\Big)\subseteq \{0\}$, $d_1,\dots,d_{\ell},\delta^{-1}\leq \cQ$, and  $|A\cap(N/9,8N/9)|\geq 3\delta N/4$, then there exists $q\leq \eta^{-(D+\epsilon)}$ such that 
\begin{equation*} \int_{\mathbf{M}'_q(\gamma)} |\widehat{f_A}(\alpha)|^2\textnormal{d}\alpha \gg_{\textnormal{\textbf{h}},\epsilon} \begin{cases} \delta^{D+1+2\epsilon}N &\text{if } D>1  \\ \delta^{2}\log^{-C}(\delta^{-1}) N & \text{if }  D=1, \ell_2+\ell_3>0 \\ \delta^2N & \text{if } D= 1, \ell_2+\ell_3=0 \text{ or } D<1 \end{cases} 
\end{equation*}
for some $C=C(\textnormal{\textbf{k}})$.
\end{lemma}

\noindent Lemma \ref{mainit} follows from Lemma \ref{L2I} and the following standard $L^2$ density increment lemma.

\begin{lemma}[Lemma 2.3 in \cite{thesis}, see also \cite{Lucier}, \cite{Ruz}] \label{dinc} Suppose $A \subseteq [1,N]$ with $|A|=\delta N$. If  $0< \theta \leq 1$ and
\begin{equation*} \int_{\mathbf{M}'_q(\gamma)}|\widehat{f_A}(\alpha)|^2\textnormal{d}\alpha \geq \theta\delta^2 N,
\end{equation*} 
then there exists an arithmetic progression 
\begin{equation*}P=\{x+\ell q : 1\leq \ell \leq L\}
\end{equation*}
with $qL \gg \min\{\theta N, \gamma^{-1}\}  $ and $|A\cap P| \geq \delta(1+\theta/32)L$.
\end{lemma}
 
\subsection*{Proof of Lemma \ref{mainit}}
 Suppose $A\subseteq [1,N]$, $|A|=\delta N$, $(A-A)\cap \Big(I(h^{d_1}_1)+\cdots+ I(h^{d_{\ell}}_{\ell})\Big)\subseteq \{0\}$, and $d_1,\dots,d_{\ell}, \delta^{-1}\leq \cQ$. If $|A\cap (N/9,8N/9)| < 3\delta N/4$, then $\max \{ |A\cap[1,N/9]|, |A\cap [8N/9,N]| \} > \delta N/8$. In other words, $A$ has density at least $9\delta/8$ on one of these intervals. Otherwise, Lemmas \ref{L2I} and \ref{dinc} apply, so in either case, letting $\eta=c_0\delta$, there exists $q\leq \eta^{-(D+\epsilon)}$ and an arithmetic progression 
\begin{equation*}P=\{x+\ell q : 1\leq \ell \leq L\}
\end{equation*}
with $qL\gg_{\textnormal{\textbf{h}},\epsilon} \delta^{D'+\epsilon} N$ and $$|A\cap P|/L \geq \begin{cases} \delta+c\delta^{D+2\epsilon}& \text{if } D>1  \\ (1+c\log^{-C}(\delta^{-1}))\delta  & \text{if } D=1, \ell_2+\ell_3>0 \\ (1+c)\delta & \text{if } D= 1, \ell_2+\ell_3=0 \text{ or } D<1\end{cases} .$$ Partitioning $P$ into subprogressions of step size $\lambda(q)$, the pigeonhole principle yields a progression 
\begin{equation*} P'=\{y+\ell \lambda(q) : 1\leq \ell \leq N'\} \subseteq P
\end{equation*}
with $N'\geq qL/2\lambda(q)$ and $|A\cap P'|/N' \geq |A\cap P|/L$. This allows us to define a set $A' \subseteq [1,N']$ by \begin{equation*} A' = \{\ell \in [1,N'] : y+\ell \lambda(q) \in A \},
\end{equation*} which satisfies $|A'|=|A\cap P'|$ and $N'\gg_{\textbf{k},\epsilon} \delta^{D'+\epsilon}N/\lambda(q) \gg_{\textnormal{\textbf{h}},\epsilon} \delta^{(D'+\epsilon)(k+1)}N$. Moreover, by Proposition \ref{inh}, $(A-A)\cap \Big(I(h^{d_1}_1)+\cdots+ I(h^{d_{\ell}}_{\ell})\Big)\subseteq \{0\}$ implies $(A'-A')\cap \Big(I(h^{\tilde{\lambda}_1(q)d_1}_1)+\cdots+ I(h^{\tilde{\lambda}_{\ell}(q)d_{\ell}}_{\ell})\Big)\subseteq \{0\}$. \qed \\

\noindent Our task for this section is now completely reduced to a proof of Lemma \ref{L2I}.

\subsection*{Proof of Lemma \ref{L2I}} Suppose $A\subseteq [1,N]$ with $|A|=\delta N$ and $(A-A)\cap I(h^{d_1}_1)+\cdots+I(h^{d_{\ell}}_{\ell})\subseteq \{0\}$.  Further, let $\eta=c_0\delta$ for an appropriately small $c_0=c_0(\textnormal{\textbf{h}},\epsilon)>0$ and let $Q=\eta^{-(D+\epsilon)}$. Since $h^{d_i}_i(H_i)\subseteq [-N/9\ell,N/9\ell]$, 
\begin{align*} \sum_{\substack{x \in \Z \\ \n\in \tilde{H}}} f_A(x)f_A(x+h(\n))&=\sum_{\substack{x \in \Z \\ \n\in \tilde{H}}} 1_A(x)1_A(x+h(\n)) -\delta\sum_{\substack{x \in \Z \\ \n\in \tilde{H}}} 1_A(x)1_{[1,N]}(x+h(\n)) \\\\ &-\delta \sum_{\substack{x \in \Z \\ \n\in \tilde{H}}} 1_{[1,N]}(x+h(\n))1_A(x) +\delta^2\sum_{\substack{x \in \Z \\ \n\in \tilde{H}}} 1_{[1,N]}(x)1_{[1,N]}(x+h(\n))  \\\\&\leq \Big(\delta^2N -2\delta|A\cap (N/9,8N/9)|\Big)|\tilde{H}|.
\end{align*}
Therefore, if $|A \cap (N/9, 8N/9)| \geq 3\delta N/4$, then by (\ref{Hbig}) we have
\begin{equation}\label{neg} \sum_{\substack{x \in \Z \\ \n\in \tilde{H}}} f_A(x)f_A(x+h(\n)) \leq -\delta^2N\tilde{M}/4.
\end{equation} 
One can easily check using (\ref{symdif}) and orthogonality of characters that 
\begin{equation}\label{orth}
\sum_{\substack{x \in \Z \\ \n\in \tilde{H}}} f_A(x)f_A(x+h(\n))=\int_0^1 |\widehat{f_A}(\alpha)|^2 S(\alpha)\textnormal{d}\alpha +O_{\textbf{h}}(N\tilde{M}/\min{\tilde{M}_i}),
\end{equation} 
where 
\begin{equation*}S_i(\alpha)= \sum_{n=1}^{M_i}e^{2\pi i h^{d_i}_i(n)\alpha}, \quad \tilde{S}_i(\alpha)= \sum_{\substack{n=1 \\ (n,W)=1}}^{M_i}e^{2\pi i h^{d_i}_i(n)\alpha}, \quad \text{and} \quad S(\alpha)=\prod_{i=1}^{\ell_1}S_i(\alpha)\prod_{j=\ell_1+1}^{\ell}\tilde{S}_j(\alpha).
\end{equation*} 
Combining (\ref{neg}) and (\ref{orth}), we have that if $\delta^{-1}\leq \cQ$ then 
\begin{equation} \label{mass}
\int_0^1 |\widehat{f_A}(\alpha)|^2|S(\alpha)|\textnormal{d}\alpha \geq \delta^2N\tilde{M}/8.
\end{equation} Letting $\gamma=\eta^{-(D'+\epsilon)}/N$, it follows from various exponential sum estimates and observations of Lucier on auxiliary polynomials that if $d_1,\dots ,d_{\ell}, \delta^{-1} \leq \cQ$, then for $\alpha \in \mathbf{M}_q(\gamma), \ q\leq Q $, we have

\begin{equation} \label{SmajII} |S(\alpha)| \ll_{\textnormal{\textbf{h}}} b(q)\tilde{M},  
\end{equation} where $$b(q)=\begin{cases}q^{-1/D} & \text{if } \ell_2+\ell_3=0 \\ C^{\omega(q)} q^{-1/D} & \text{if }\ell_2+\ell_3>0 \end{cases}, $$ $\omega(q)$ is the number of distinct prime factors of $q$, and $C=C(\textbf{k})$. Further, for $\alpha \in \mathfrak{m}(\gamma,Q)$ we have 
\begin{equation} \label{SminII} |S(\alpha)| \leq  \delta \tilde{M}/16 ,
\end{equation} provided $c_0$ was chosen sufficiently small. Details of these estimates are provided in Appendix \ref{apxA}.

\noindent From (\ref{SminII}) and Plancherel's Identity, we have \begin{equation*}  \int_{\mathfrak{m}(\gamma,Q)} |\widehat{f_A}(\alpha)|^2|S(\alpha)|\textnormal{d}\alpha \leq \delta^2N\tilde{M}/16, \end{equation*} which together with (\ref{mass}) yields \begin{equation}\label{majmass}  \int_{\mathfrak{M}(\gamma,Q)}|\widehat{f_A}(\alpha)|^2|S(\alpha)|\textnormal{d}\alpha \geq \delta^2 N\tilde{M}/16. \end{equation} 
From (\ref{SmajII}) and (\ref{majmass}) , we have 
\begin{equation} \label{majmassII} \sum_{q=1}^Q  b(q)\int_{\mathbf{M}_q(\gamma)}|\widehat{f_A}(\alpha)|^2 {d}\alpha \gg_{\textnormal{\textbf{h}}} \delta^2N.
\end{equation}
If $D\geq 1$, then the function $b(q)$ satisfies $b(qr)\geq b(r)/q$, and we make use of the following proposition.
\begin{proposition}\label{rstrick} For any $\gamma,Q>0$ satisfying $2\gamma Q^2<1$ and any function $b: \N \to [0,\infty)$ satisfying $b(qr)\geq b(r)/q$, we have $$\max_{q\leq Q} \int_{\mathbf{M}'_q(\gamma)}|\widehat{f_A}(\alpha)|^2 {d}\alpha \geq Q \Big(2\sum_{q=1}^Q qb(q)\Big)^{-1} \sum_{r=1}^Q b(r)\int_{\mathbf{M}_r(\gamma)}|\widehat{f_A}(\alpha)|^2 {d}\alpha. $$
\end{proposition}

\begin{proof} By (\ref{majdisj}) we have \begin{align*} \Big(\sum_{q=1}^Q qb(q)\Big) \max_{q\leq Q} \int_{\mathbf{M}'_q(\gamma)}|\widehat{f_A}(\alpha)|^2 {d}\alpha &\geq \sum_{q=1}^Q qb(q) \int_{\mathbf{M}'_q(\gamma)}|\widehat{f_A}(\alpha)|^2 {d}\alpha \\\\ &= \sum_{q=1}^Q qb(q) \sum_{r|q} \int_{\mathbf{M}_r(\gamma)}|\widehat{f_A}(\alpha)|^2 {d}\alpha \\\\ &= \sum_{r=1}^Q \int_{\mathbf{M}_r(\gamma)}|\widehat{f_A}(\alpha)|^2 {d}\alpha \sum_{q=1}^{Q/r} qrb(qr) \\\\ &\geq \frac{Q}{2}\sum_{r=1}^Q b(r)\int_{\mathbf{M}_r(\gamma)}|\widehat{f_A}(\alpha)|^2 {d}\alpha,
\end{align*} 

\noindent where the last inequality comes from replacing $b(qr)$ with $b(r)/q$, and the proposition follows.
\end{proof}
\noindent Using the known estimate

\begin{equation*} \sum_{q=1}^Q C^{\omega(q)} \ll_C Q\log^C Q \quad \text{(see \cite{Selberg})}
\end{equation*}
for any $C>0$, we see that 
\begin{equation} \label{bqavg} Q^{-1}\sum_{q=1}^Q qb(q) \ll_{\textbf{k},\epsilon} \begin{cases} 1 & \text{if } D=1, \ell_2+\ell_3=0 \\ \log^C(\delta^{-1}) &\text{if } D=1, \ell_2+\ell_3>0 \\ \delta^{-(D-1+2\epsilon)} &\text{if } D>1\end{cases},
\end{equation}
which combined with (\ref{majmassII}) and Proposition \ref{rstrick} establishes the lemma for $D\geq 1$.   If $D<1$, then the lemma follows from (\ref{majmassII}) and the fact that $\displaystyle{\sum_{q=1}^{\infty} b(q)}$ converges. \qed     

\section{S\'ark\"ozy's Method with Prime Inputs: Theorem \ref{main2} for $D'>2$ and $D'\leq 1$} \label{primesec}
Structurally speaking, the arguments for Theorem \ref{main2} are essentially the same as those for Theorem \ref{main1}, with careful adaptations required to account for our somewhat limited understanding of the distribution of primes in arithmetic progressions. 
\subsection{Counting primes in arithmetic progressions} 
  
For $x,a,q\in \N$, we define 
\begin{equation*} \psi(x,a,q) = \sum_{\substack{p\leq x \\ p\equiv a \text{ mod } q}} \log p,
\end{equation*} 
where the sum is taken over primes. The classical estimates on $\psi(x,a,q)$ come from the famous Siegel-Walfisz Theorem, which can be found for example in Corollary 11.19 of \cite{MV}.
\begin{lemma}[Siegel-Walfisz Theorem] \label{SW} If $q\leq (\log x)^B$, and $(a,q)=1$, then $$ \psi(x,a,q) = x/\phi(q) + O(xe^{-c\sqrt{\log x}}) $$ for some constant $c=c(B)>0$.
\end{lemma}

\noindent Ruzsa and Sanders \cite{Ruz} established asymptotics for $\psi(x,a,q)$ for certain moduli $q$ beyond the  limitations of Lemma \ref{SW} by exploiting a dichotomy based on exceptional zeros, or lack thereof, of Dirichlet $L$-functions. In particular, the following result follows from their work.

\begin{lemma} \label{RS} For any $\cQ,B>0$, there exist $q_0 \leq \cQ^{B}$ and $\rho \in [1/2,1)$ with $(1-\rho)^{-1} \ll q_0$ such that 

\begin{equation}\label{lb} \psi(x,a,q)= \frac{x}{\phi(q)}-\frac{\chi(a)x^{\rho}}{\phi(q)\rho}+O\Big(x \exp\Big(-\frac{c\log x}{\sqrt{\log x}+B^2\log \cQ}\Big)B^2\log \cQ\Big),
\end{equation}

\noindent where $\chi$ is a Dirichlet character modulo $q_0$, provided $q_0 \mid q$, $(a,q)=1$,  and $q \leq (q_0\cQ)^{B}$.
\end{lemma}  

\noindent Lemma \ref{RS} is a purpose-built special case of Proposition 4.7 of \cite{Ruz}, which in the language of that paper can be deduced by considering the pair $(\cQ^{B^2+B},\cQ^{B})$, where $q_0$ is the modulus of the exceptional Dirichlet character if the pair is exceptional and $q_0=1$ if the pair is unexceptional. 

\noindent It is a calculus exercise to verify that if $\beta\in [0,1/2]$ and $x\geq 16$, then $1-x^{-\beta}/(1-\beta)\geq \beta,$ which implies that the main term in Lemma \ref{RS} satisfies 
\begin{equation}\label{x/q0}\Re\Big((x-\chi(a)x^{\rho}/\rho)/\phi(q)\Big) \geq (1-\rho) x/\phi (q) \gg x/q_0\phi(q). \end{equation} 

\subsection{Main iteration lemma} For the remainder of this section, we fix nonzero $\P$-intersective polynomials $h_1,\dots,h_{\ell}\in \Z[x]$.
Unlike in Section \ref{unr}, we also must fix at the outset a natural number $N$, in order to carefully apply estimates on $\psi(x,a,q)$. 

\noindent Specifically, we let $\cQ=e^{c_1\sqrt{\log N}}$ for a sufficiently small constant $c_1=c_1(\textbf{k})>0$, and we apply Lemma \ref{RS} with $B=10K$, letting $q_0\leq \cQ^{10K}$, $\rho \in [1/2,1)$, and the Dirichlet character $\chi$ be as in the conclusion. 

\noindent We see that if $c_1$ is sufficiently small and $X\geq N^{1/10k}$, then  
\begin{equation} \label{PsiAsymPI} \psi(x,a,q)= \frac{x}{\phi(q)}-\frac{\chi(a)x^{\rho}}{\phi(q)\rho}+O(X\cQ^{-1000K^2})
\end{equation} 
for all $x\leq X$, provided $q_0 \mid q$, $(a,q)=1$, and $q\leq(q_0\cQ)^{10K}$.

\noindent We deduce Theorem \ref{main2} (outside of $1<D'\leq2$) from the following analog of Lemma \ref{mainit}.

\begin{lemma} \label{main5} Suppose $A\subseteq [1,L]$ with $|A|=\delta L$ and $L\geq \sqrt{N}$.  If  $$(A-A)\cap \Big(\V_1^{d_1}(h^{d_1}_1)+\cdots+\V^{d_{\ell}}_{\ell}(h^{d_{\ell}}_{\ell})\Big)\subseteq \{0\},$$ $q_0 \mid d_i$ and $d_i/q_0,\delta^{-1} \leq \cQ,$ for $1\leq i \leq \ell$, then there exists $q\ll_{\textnormal{\textbf{h}},\epsilon} \delta^{-(D'+\epsilon)}$ and $A'\subseteq [1,L']$  with $L' \gg_{\textnormal{\textbf{h}}.\epsilon} \delta^{(D'+\epsilon)(k+1)} L,$ \begin{equation*}\frac{|A'|}{L'} \geq \begin{cases} \delta+c\delta^{D'+2\epsilon}& \text{if } D'>1  \\ (1+c\log^{-C}(\delta^{-1}))\delta  & \text{if } D'=1, \ell>1 \\ (1+c)\delta & \text{if } D'<1\text{ or } \ell=k_1=1\end{cases} , \end{equation*} and $$(A'-A')\cap \Big(\V^{\tilde{\lambda}_1(q)d_1}_1(h^{\tilde{\lambda}_1(q)d_1}_1)+\cdots+ \V^{\tilde{\lambda}_{\ell}(q)d_{\ell}}_{\ell}(h^{\tilde{\lambda}_{\ell}(q)d_{\ell}}_{\ell})\Big)\subseteq \{0\}$$ for some $c=c(\textnormal{\textbf{h}},\epsilon)>0$ and $C=C(\textnormal{\textbf{k}})$.\end{lemma}  

\subsection*{Proof that Lemma \ref{main5} implies Theorem \ref{main2} for $D'>2$ and $D'\leq 1$} Suppose $A\subseteq [1,N]$ with $|A|=\delta N$ and $$a-a'\neq \sum_{i=1}^{\ell} h_i(p_i)$$ for all distinct pairs $a,a'\in A$ and for all primes $p_1,\dots, p_{\ell}$ with $h_1(p_1),\dots, h_{\ell}(p_{\ell}) \neq 0$. In particular, this implies that $$(A-A)\cap \Big(\V_1^{1}(h^{1}_1)+\cdots+\V^{1}_{\ell}(h^{1}_{\ell})\Big)\subseteq \{0\}.$$ Partitioning $[1,N]$, the pigeonhole principle guarantees the existence of an arithmetic progression $$P=\{x+\ell \lambda(q_0) : 1\leq \ell \leq N_0\}\subseteq [1,N]$$ with $N_0\geq N/2\lambda(q_0)$ and $|A\cap P|\geq \delta N_0$. Defining $A_0\subseteq [1,N_0]$ by $$A_0=\{\ell \in [1,N_0] : x+\ell \lambda(q_0) \in A\},$$ we see that $|A_0|\geq \delta N_0$ and $$(A_0-A_0)\cap \Big(\V^{\tilde{\lambda}_1(q_0)}_1(h^{\tilde{\lambda}_1(q_0)}_1)+\cdots+ \V^{\tilde{\lambda}_{\ell}(q_0)}_{\ell}(h^{\tilde{\lambda}_{\ell}(q_0)}_{\ell})\Big)\subseteq \{0\}.$$
After this initial passage to a subprogression, Theorem \ref{main2} (for $D'>2$ and $D'\leq 1$) follows from Lemma \ref{main5} in a manner completely analogous to the deduction of Theorem \ref{main1} from Lemma \ref{mainit}.  \qed   

\noindent We establish Lemma \ref{main5} from the following analog of Lemma \ref{L2I}.
\begin{lemma}\label{PIL2} Suppose $A\subseteq [1,L]$ with $|A|=\delta L$ and $L\geq \sqrt{N}$, let $\eta=c_0\delta$ for a sufficiently small constant $c_0=c_0(\textnormal{\textbf{h}},\epsilon)>0$, and let $\gamma=\eta^{-(D'+\epsilon)}/L$. If $q_0\mid d_i$, $d_i/q_0, \delta^{-1} \leq \cQ$, $$(A-A)\cap \Big(\V_1^{d_1}(h^{d_1}_1)+\cdots+\V^{d_{\ell}}_{\ell}(h^{d_{\ell}}_{\ell})\Big)\subseteq \{0\},$$ and $|A\cap(L/9,8L/9)|\geq 3\delta L/4$, then there exists $q \leq \eta^{-(D'+\epsilon)}$ such that 
\begin{equation*} \int_{\mathbf{M}_q(\gamma)} |\widehat{f_A}(\alpha)|^2\textnormal{d}\alpha \gg_{\textnormal{\textbf{h}}, \epsilon} \begin{cases} \delta^{D'+1+2\epsilon}L &\text{if } D'>1  \\ \delta^{2}\log^{-C}(\delta^{-1}) L & \text{if }  D'=1, \ell>1 \\ \delta^2L & \text{if } D'<1 \text{ or } \ell=k_1=1 \end{cases} 
\end{equation*} 
for some $C=C(\textnormal{\textbf{k}})$.
\end{lemma}

\noindent The deduction of Lemma \ref{main5} from Lemma \ref{PIL2} is effectively identical to the deduction of Lemma \ref{mainit} from Lemma \ref{L2I}, and our task for this section is now reduced to a proof of Lemma \ref{PIL2}.

\subsection*{Proof of Lemma \ref{PIL2}}
The proof of Lemma \ref{PIL2} is  analogous to that of Lemma \ref{L2I}. Suppose $A\subseteq [1,L]$ with $|A|=\delta L$, $L\geq \sqrt{N}$, $$(A-A)\cap \Big(\V_1^{d_1}(h^{d_1}_1)+\cdots+\V^{d_{\ell}}_{\ell}(h^{d_{\ell}}_{\ell})\Big)\subseteq \{0\},$$ $q_0\mid d_i$, and $d_i/q_0,\delta^{-1}\leq \cQ$. We make liberal use here of notation defined in Section \ref{prelim}, defining the relevant objects in terms of $L$, the size of the current ambient interval, as opposed to the previously fixed $N$. 

\noindent For $1\leq i \leq \ell$, we define functions $\nu_i$ on $\Z$ \begin{equation*} \nu_i(n)=\frac{\phi(d_i)}{d_i}\log(d_in+r_i^{d_i})1_{\Lambda_i^{d_i}}(n),
\end{equation*} and for $\n\in \Z^{\ell}$ we let $\nu(\n)=\prod_{i=1}^{\ell}\nu_i(n_i).$ 

\noindent Just as in the derivation of (\ref{mass}), we have by (\ref{symdif})  and orthogonality of characters that 
\begin{equation*} \int_0^1|\widehat{f_A}(\alpha)|^2|\S(\alpha)| \textnormal{d}\alpha \geq \Big| \sum_{\substack{x\in \Z \\ \n\in H}} f_A(x)f_A(x+h(\n))\nu(\n)\Big| +O_{\textbf{h}}\Big(\frac{LM\log L}{\min M_i}\Big) \end{equation*}
and hence 
\begin{equation}\label{massPI} \int_0^1|\widehat{f_A}(\alpha)|^2|\S(\alpha)| \textnormal{d}\alpha \geq\delta^2L\Psi/2+O_{\textbf{h}}\Big(\frac{LM\log L}{\min M_i}\Big), 
\end{equation} 
where  $$\S_i(\alpha)=\sum_{n=1}^{M_i}\nu_i(n)e^{2\pi i h^{d_i}_{i}(n)\alpha},$$ $$\S(\alpha)=\prod_{i=1}^{\ell}\S_i(\alpha), $$ and $$\Psi=\prod_{i=1}^{\ell}\frac{\phi(d_i)}{d_i}\psi(d_iM_i+r^i_{d_i}, r^i_{d_i}, d_i).$$  From (\ref{PsiAsymPI}) and (\ref{x/q0}), we know that 
\begin{equation} \label{m/q3K} \Psi \geq \prod_{i=1}^{\ell}c(1-\rho)M_i \geq (Cq_0)^{-\ell}M \geq \cQ^{-20K\ell}M,\end{equation} which combined with (\ref{massPI}) implies
\begin{equation}\label{massPI2}\int_0^1|\widehat{f_A}(\alpha)|^2|\S(\alpha)|\textnormal{d}\alpha \geq \delta^2L\Psi/4. 
\end{equation} 
We let $\eta=c_0\delta$ for a sufficiently small constant $c_0=c_0(\textbf{h},\epsilon)>0$, $Q=\eta^{-(D'+\epsilon)}$, and $\gamma=Q/L$. It then follows from various exponential sum estimates, observations of Lucier on auxiliary polynomials, and Theorem 4.1 of \cite{lipan} that if $\alpha \in \mathbf{M}_q(\gamma), \ q\leq Q$, then
\begin{equation} \label{WmajPI} |\S(\alpha)| \ll_{\textbf{h}} \begin{cases}  \Psi/\phi(q) & \text{if } \ell=k_1=1 \\C^{\omega(q)}q^{-1/D'}(q/\phi(q))^{\ell}\Psi & \text{else} \end{cases} ,
\end{equation} 
where $\omega(q)$ is the number of distinct prime factors of $q$ and $C=C(\textbf{k})$, and 
\begin{equation} \label{WminPI} |\S(\alpha)| \leq \delta \Psi/8 \quad \text{for all} \quad \alpha \in \mathfrak{m}(\gamma, Q),
\end{equation} provided we choose $c_0$ sufficiently small. We discuss these estimates in more detail in Appendix \ref{apxB}, and the remainder of the proof is completely analogous to that of Lemma \ref{L2I}. \qed


\section{Double Iteration Method: Theorem \ref{main1} for $1<D\leq 2$} \label{Sec3}

In this section, we apply an adapted version of a double iteration argument, developed by Pintz, Steiger, and Szemer\'edi \cite{PSS} and previously modified and streamlined in \cite{BPPS} and \cite{HLR}. A technical difference with the previous method is the necessity that both the space and frequency domains be discrete.

\subsection{Fourier analysis and the circle method on $\Z/N\Z$}
We identify  subsets of the interval $[1,N]$ with subsets of the finite group $\Z_N=\Z/N\Z$, on which we utilize the normalized discrete Fourier transform. Specifically, for a function $F: \Z_N \to \C$, we define $\widehat{F}: \Z_N \to \C$ by \begin{equation*} \widehat{F}(t) = \frac{1}{N} \sum_{x \in \Z_N} F(x)e^{-2 \pi ixt/N}. \end{equation*} 

\noindent We make the analogous definitions for the major and minor arcs on $\Z_N$, singling out the zero frequency rather than introducing the balanced function.

\begin{definition}\label{arcs} Given $N\in \N$ and $K,Q>0$, we define, for each $q\in \N$ and $a \in [1,q]$, 
\begin{equation*} \mathbf{M}_{a,q}(K)=  \left\{t \in \Z_N: \left|\frac{t}{N}-\frac{a}{q}\right| < \frac{K}{N} \right\} \text{ \ and \ } \mathbf{M}_q(K) = \bigcup_{(a,q)=1}\mathbf{M}_{a,q}(K)\setminus \{0\}. \end{equation*} 
We then define $\mathfrak{M}(K,Q)$, the \emph{major arcs}, by \begin{equation*} \mathfrak{M}(K,Q) = \bigcup_{q=1}^{Q} \mathbf{M}_q(K), \end{equation*} and $\mathfrak{m}(K,Q)$, the \emph{minor arcs}, by $\mathfrak{m}(K,Q) = \Z_N \setminus (\mathfrak{M}(K,Q) \cup \{0\})$. It is important to note that as long as $2KQ^2<N$, we have that $\mathbf{M}_{a,q} \cap \mathbf{M}_{b,r} = \emptyset$ whenever $a/q \neq b/r$, $q,r \leq Q$.
\end{definition} 

\noindent We note that the sets defined above certainly depend on $N$, despite its absence from the notation. In practice, $N$ should always be replaced with the size of the appropriate ambient group, often denoted in the intermediate stages of the iterations by $L$.

\subsection{Overview of the argument}  
In a manner essentially identical to our attainment of (\ref{majmassII}) in Section \ref{unr}, we begin by observing that if $(A-A)\cap \Big(I(h_1)+\cdots+I(h_{\ell})\Big)\subseteq \{0\}$ for a set $A\subseteq [1,N]$ and intersective polynomials $h_1,\dots, h_{\ell}\in \Z[x]$, one can apply the circle method and Weyl sum estimates to show that this unexpected behavior implies substantial $L^2$  mass of $\widehat{A}$ over nonzero frequencies near rationals with small denominator. 

\noindent At this point, the traditional method, which we employed in Section \ref{unr}, is to use the pigeonhole principle to conclude that there is one single denominator $q$ such that $\widehat{A}$ has $L^2$ concentration around  rationals with denominator $q$. From this information, one can conclude that $A$ has increased density on a long arithmetic progression with step size an appropriate multiple of $q$, leading to a new denser set with an inherited lack of structure  and continued iteration. 

\noindent Pintz, Steiger, and Szemer\'edi \cite{PSS}  observed that pigeonholing to obtain a single denominator $q$ is a potentially wasteful step. We follow their approach, observing the following dichotomy: 
\begin{quotation}
\noindent \textbf{Case 1.} There is a single denominator $q$ such that $\widehat{A}$ has extremely high $L^2$ concentration, greater than yielded by the pigeonhole principle, around rationals with denominator $q$. This leads to a very large density increment on a long arithmetic progression. 
\medskip
 
\noindent \textbf{Case 2.} The $L^2$ mass of $\widehat{A}$ on the major arcs is spread over many denominators. In this case, an iteration procedure using the ``combinatorics of rational numbers" can be employed to build a large collection of frequencies at which $\widehat{A}$ is large, then Plancherel's identity is applied to bound the density of $A$.
\end{quotation}

\noindent Philosophically, Case 1 provides more structural information about the original set $A$ than Case 2 does. The downside is that the density increment procedure yields a new set and potentially  a new polynomial, while the iteration in Case 2 leaves these objects fixed. 
With these cases in mind, we can now outline the argument, separated into two distinct phases. 
\begin{quotation}
\noindent \textbf{Phase 1} (The Outer Iteration): Given a set $A$ and intersective polynomials $h_1,\dots,h_{\ell}\in \Z[x]$, ordered and partitioned as in Section \ref{unr}, with $(A-A)\cap \Big(I(h_1)+\cdots+I(h_{\ell})\Big)\subseteq \{0\},$ we ask if the set falls into Case 1 or Case 2 described above. If it falls into Case 2, then we proceed to Phase 2. 
\medskip

\noindent
If it falls into Case 1, then the density increment procedure yields a new subset $A_1$ of a slightly smaller interval with significantly greater density, and  $$(A_1-A_1) \cap \Big(I(h^{\tilde{\lambda}_1(q)}_1)+\cdots+I(h^{\tilde{\lambda}_{\ell}(q)}_{\ell})\Big)\subseteq \{0\}.$$
We can then iterate this process as long as the resulting interval is not too small, and the dichotomy holds as long as the coefficients of the corresponding polynomials are not too large. We show that if the resulting sets remain in Case 1, and the process iterates until the interval shrinks down or the coefficients grow to the limit, then the density of the original set $A$ must have satisfied a bound stronger than the one purported in Theorem \ref{main1} for $1 <D \leq 2$. \medskip

\noindent
Contrapositively, we assume that the original density does not satisfy this stricter bound, and we conclude that one of the sets yielded by the density increment procedure must lie in a large interval, have no nonzero differences in $I(h^{d_1}_1)+\cdots+ I(h^{d_{\ell}}_{\ell})$ for reasonably small $d_1,\dots, d_{\ell}$, and fall into Case 2. We call that set $B\subseteq[1,L]$. 
\medskip
\end{quotation}

\noindent We now have a set $B \subseteq [1,L]$ with $(B-B)\cap \Big(I(h^{d_1}_1)+\cdots+ I(h^{d_{\ell}}_{\ell})\Big)\subseteq \{0\}$ which falls into Case 2, so we can adapt the strategy of \cite{PSS}, \cite{BPPS}, and \cite{HLR}. It is in this phase that we use that $D\leq 2$.

\begin{quotation} \textbf{Phase 2} (The Inner Iteration): We prove that given a frequency $s\in \Z_L$ with $s/L$ close to a rational $a/q$ such that $\widehat{B}(s)$ is large, there are lots of nonzero frequencies $t\in \Z_L$ with $t/L$ close to rationals $b/r$ such that $\widehat{B}(s+t)$ is almost as large. This intuitively indicates that a set $P$ of frequencies associated with large Fourier coefficients can be blown up to a much larger set $P'$ of frequencies associated with nearly as large Fourier coefficients. 
\medskip 

\noindent
The only obstruction to this intuition is the possibility that there are many pairs $(a/q,b/r)$ and $(a'/q',b'/r')$ with $a/q+b/r = a'/q'+b'/r'$. Observations made in \cite{PSS} and \cite{BPPS} on the combinatorics of rational numbers demonstrate that this potentially harmful phenomenon can not occur terribly often. 
\medskip

\noindent
Starting with the trivially large Fourier coefficient at $0$, this process is applied as long as certain parameters are not too large, and the number of iterations is ultimately limited by the growth of the divisor function. Once the iteration is exhausted, we use the resulting set of large Fourier coefficients and Plancherel's Identity to get the upper bound on the density of $B$, which is by construction larger than the density of the original set $A$, claimed in Theorem \ref{main1} for $1<D\leq 2$.
\end{quotation}

\subsection{Reduction to two key lemmas} 

For the remainder of this section, we fix intersective polynomials $h_1,\dots, h_{\ell} \in \Z[x]$, partitioned as indicated in Theorem \ref{main1}. Namely, $h_1,\dots,h_{\ell_1}$ are arbitrary intersective polynomials, $h_{\ell_1+1},\dots, h_{\ell_1+\ell_2}$ are nonconstant monomials, and $h_{\ell_1+\ell_2+1}, \dots, h_{\ell}$ are nonmonomials with no constant term. We let $$D=\Big(\sum_{i=1}^{\ell_1} k_i^{-1}+\ell_2/2+\sum_{i=\ell_1+\ell_2+1}^{\ell} r_i^{-1}\Big)^{-1},$$ where $k_i=\deg(h_i)$ and $r_i$ is the number of nonzero coefficients of $h_i$, and again we let $D'=\sum_{i=1}^{\ell} k_i^{-1}$.

\noindent Further, we let $\textbf{k}=(k_1,\dots,k_{\ell})$, $k=\prod_{i=1}^{\ell}k_i,$ and $\rho=2^{-10k}$.  We also fix a natural number $N$ and an arbitrary $\epsilon>0$ and let $\cQ=(\log N)^{\epsilon \log\log\log N}.$ We deduce Theorem \ref{main1} (for $1<D\leq 2$) from two key lemmas, corresponding to the two phases outlined in the overview, the first of which yields a set with substantial Fourier $L^2$ mass distributed over rationals with many small denominators. 

\begin{lemma} \label{outer} Suppose $A\subseteq [1,N]$ with $|A|=\delta N$ and $(A-A)\cap \Big(I(h_1)+\cdots+I(h_{\ell})\Big)\subseteq \{0\}$. If \begin{equation} \label{behrend} \delta \geq e^{-(\log N)^{\epsilon/4}},\end{equation} then there exists  $B \subseteq [1,L]$ and $d_1,\dots, d_{\ell} \leq N^{\rho/2}$ satisfying $L \geq N^{1-\rho}$, $|B|/L=\sigma \geq \delta$, and $$(B-B) \cap \Big(I(h^{d_1}_1)+\cdots+I(h^{d_{\ell}}_{\ell})\Big)\subseteq \{0\}.$$ Further, $B$ satisfies $|B\cap[1,L/2]|\geq \sigma L/3$ and  \begin{equation} \label{maxmass} \max_{q\leq \cQ}\sum_{t\in \mathbf{M}_q(\cQ)}  |\widehat{B}(t)|^2 \leq \sigma^2(\log N)^{-1+\epsilon}. \end{equation} 
\end{lemma}  

\noindent The second lemma corresponds to the iteration scheme in which a set of large Fourier coefficients from distinct major arcs is blown up in such a way that the relative growth of the size of the set is much greater than the relative loss of pointwise mass.

\begin{lemma} \label{itn} Suppose $B\subseteq [1,L]$ and $d_1\dots,d_{\ell}$ are as in the conclusion of Lemma \ref{outer}, let $B_1=B\cap[1,L/2]$, let $m=2(\min_{i=1}^{\ell}k_i)^2-1$, and suppose $\sigma \geq \cQ^{-1/m}$. If $D\leq 2$, then given $U,V,K \in \N$ with $\max \{U,V,K\}\leq \cQ^{1/m}$ and a set \begin{equation*} P \subseteq \left\{t \in \bigcup_{q=1}^{V} \mathbf{M}_q(K)\cup\{0\} : |\widehat{B_1}(t)|  \geq \frac{\sigma}{U} \right\} \end{equation*} satisfying \begin{equation} \label{dis}| P \cap \mathbf{M}_{a,q}(K)| \leq 1 \text{ \ whenever \ } q \leq V,\end{equation} there exist $U',V',K' \in \N$ with $\max\{U',V',K'\} \ll_{\textnormal{\textbf{k}}} (\max \{U,V,K\})^{m/2}\sigma^{-(2D+m)/2}$ and a set  \begin{equation} \label{P'1} P' \subseteq \left\{t \in \bigcup_{q=1}^{V'} \mathbf{M}_q(K')\cup\{0\} : |\widehat{B_1}(t)|  \geq \frac{\sigma}{U'} \right\}\end{equation} satisfying \begin{equation} \label{P'2} | P' \cap \mathbf{M}_{a,q}(K')| \leq 1 \text{ \ whenever \ } q\leq V'\end{equation} and \begin{equation} \label{P'3} \frac{|P'|}{(U')^2} \geq \frac{|P|}{U^2}(\log N)^{1-15\epsilon}.\end{equation}
\end{lemma} 
 
\subsection{Proof of Theorem \ref{main1} for $1<D\leq 2$} \label{hardproof} In order to establish (\ref{PSSbound}), we can assume that \begin{equation*}\delta \geq (\log N)^{-\log\log\log\log N}. \end{equation*} Therefore, Lemma \ref{outer} produces a set $B$ of density $\sigma \geq \delta$ with the stipulated properties, and we set $P_0=\{0\}$, $U_0=3$, and $V_0=K_0=1$. Then, if $D\leq 2$, Lemma \ref{itn} yields, for each $n$, a set $P_n$ with parameters $U_n,V_n,K_n$ such that \begin{equation*}\max\{U_n,V_n,K_n\} \leq (\log N)^{(m/2)^{n+3}\log\log\log\log N} \end{equation*} and \begin{equation*}\frac{1}{\sigma} \geq \frac{1}{\sigma^2}\sum_{t\in P_n} |\widehat{B_1}(t)|^2 \geq \frac{|P_n|}{U_n^2} \gg (\log N)^{n(1-15\epsilon)}, \end{equation*} where the left-hand inequality comes from Plancherel's Identity, as long as $\max\{U_n,V_n,K_n\} \leq \cQ^{1/m}$. This holds with $n=(1-\epsilon)(\log\log\log\log N) /\log (m/2)$, as $(m/2)^{n+3} \leq (\log\log\log N)^{1-\epsilon/2}$, and desired case of Theorem \ref{main1} follows. 
\qed

\subsection{The Outer Iteration} We begin the first phase with the following discrete analog of Lemma \ref{dinc}, the proof of which is effectively identical, and versions of which can be found in \cite{Lucier} and \cite{thesis}.

\begin{lemma} \label{dinczn} Suppose $B\subseteq [1,L]$ with $|B|=\sigma L$. If $0<\theta \leq 1$ and
$$\sum_{t\in \mathbf{M}'_q(K)}|\widehat{B}(t)|^2 \geq \theta\sigma^2,$$ then there exists an arithmetic progression \begin{equation*}P=\{x+\ell q : 1\leq \ell \leq L'\}
\end{equation*} with $qL' \gg \min\{\theta, K^{-1}\}L$ and $|B\cap P|\geq \sigma(1+\theta/32)L'$. 
\end{lemma}

\noindent Lemma \ref{dinczn} and Proposition \ref{inh} immediately combine to yield the following iteration lemma, corresponding to Case 1 discussed in the overview, from which we deduce Lemma \ref{outer}.

\begin{lemma} \label{inc} Suppose $B\subseteq [1,L]$ with $|B|=\sigma L$ and $(B-B) \cap \Big(I(h^{d_1}_1)+\cdots+I(h^{d_{\ell}}_{\ell})\Big)\subseteq \{0\}$. If  
\begin{equation}\label{Bmass} \sum_{t\in \mathbf{M}_q(\cQ)} |\widehat{B}(t)|^2 \geq \sigma^2(\log N)^{-1+\epsilon},\end{equation}
 for some $q \leq \cQ$, then there exists $B'\subseteq [1,L']$ satisfying $L' \gg L/\cQ^{k+1},$ $$(B'-B')\cap  \Big(I(h^{\tilde{\lambda}_1(q)d_1}_1)+\cdots+ I(h^{\tilde{\lambda}_{\ell}(q)d_{\ell}}_{\ell})\Big)\subseteq \{0\},$$ and
\begin{equation*} 
|B'|/L'\geq  \sigma(1+(\log N)^{-1+\epsilon}/32).\end{equation*}

\end{lemma} 

\subsection*{Proof of Lemma \ref{outer}}

Setting $A_0$=$A$, $N_0=N$, $\delta_0=\delta$, and $d_1^0,\dots,d_{\ell}^0=1$, we iteratively apply Lemma \ref{inc}. This yields, for each $j$, a set $A_j \subseteq [1,N_j]$ with $|A_j|=\delta_jN_j$ and $$(A_j-A_j)\cap \Big(I(h^{d^j_1}_1)+\cdots+I(h^{d^j_{\ell}}_{\ell})\Big)\subseteq \{0\},$$satisfying
\begin{equation} \label{Ndeld} N_j \geq N/(C\cQ)^{(k+1)j}, \quad \delta_j \geq \delta_{j-1}(1+(\log N)^{-1+\epsilon}/32),  \quad d_j \leq \cQ^{kj},
\end{equation}
where $C$ is an absolute constant, as long as either
\begin{equation}\label{newmass}\max_{q\leq \cQ}\sum_{t\in \mathbf{M}_q(\cQ)} |\widehat{A_j}(t)|^2\geq \delta_j^2(\log N)^{-1+\epsilon}
\end{equation}
or $|A_j\cap[1,N_j/2]|<\delta_jN_j/3$, as the latter condition implies $A_j$ has density at least $3\delta_j/2$ on the interval $(N_j/2,N_j]$. We see that by (\ref{behrend}) and (\ref{Ndeld}), the density $\delta_j$ will exceed 1 after \begin{equation*}64\log(\delta^{-1})(\log N)^{1-\epsilon} \leq (\log N)^{1-\epsilon/2}\end{equation*} steps, hence  (\ref{newmass}) fails and $|A_j\cap[1,N_j/2]|\geq \delta_jN_j/3$ for some 
\begin{equation} \label{kbound} j\leq (\log N)^{1-\epsilon/2}.\end{equation}
However, we see that (\ref{behrend}), (\ref{Ndeld}), and (\ref{kbound}) imply 
 \begin{equation*}N_j \geq N/(C\cQ)^{(k+1)(\log N)^{1-\epsilon/2}} \geq Ne^{-(\log N)^{1-\epsilon/4}} \geq N^{1-\rho}, \end{equation*}
so we set $B=A_j$, $L=N_j$, $\sigma = \delta_j$, and $d_i=d_i^j$ for $1\leq i \leq \ell$, and we see further that
\begin{equation*} d_i \leq \cQ^{k(\log N)^{1-\epsilon/2}}\leq e^{(\log N)^{1-\epsilon/4}} \leq N^{\rho/2}, \end{equation*} 
as required.
\qed

\noindent The task of this section is now reduced to a proof of Lemma \ref{itn}.

\subsection{The Inner Iteration: Proof of Lemma \ref{itn}} Let $B\subseteq [1,L]$ and $d_1,\dots,d_{\ell} \in \N$ be as in the conclusion of Lemma \ref{outer}, let $B_1=B\cap[1,L/2]$, let $m=2(\min\{k_i\})^2-1$, and suppose $\sigma =|B|/L \geq \cQ^{-1/m}$. For $1\leq i \leq \ell$, let $M_i=\lfloor (L/3\ell|b_{i}|)^{1/k_i}\rfloor$, where $b_{i}$ is the leading coefficient of $h^{d_i}_i$. Letting $$H_{i}=\begin{cases}\{n \in \N : 0<h^{d_i}_i(n)<L/3\ell\} & \text{if } b_i>0 \\\{n \in \N : -L/3\ell<h^{d_i}_i(n)<0\} & \text{if } b_i<0\end{cases},$$  we note that by (\ref{symBIG}) we have \begin{equation}\label{symdif2} |H_{i} \ \triangle  \ [1,M_i]| \ll_{h_i} 1.  \end{equation} for $1\leq i \leq \ell$. Suppose we have a set $P$ with parameters $U,V,K$ as specified in the hypotheses of Lemma \ref{itn}, fix an element $s \in P$, let $\eta=c_0\sigma/U$ for a sufficiently small constant $c_0=c_0(\textbf{k},\epsilon)>0$, let $W=\prod_{p\leq \eta^{-(D'+\epsilon)}}p$ where the product is taken over primes, and let $w=\prod_{p\leq \eta^{-(D'+\epsilon)}} (1-1/p).$ As in previous sections, we let $h(\n)=h^{d_1}_1(n_1)+\cdots+h^{d_{\ell}}_{\ell}(n_\ell)$, $Z=\{\n \in \N^{\ell} : n_i\in H_i, \  h(\n)=0\}$, $\tilde{H_i}=\{n\in H_i : (n,W)=1\}$,  and $\tilde{H}=\Big(H_1\times \cdots \times H_{\ell_1}\times \tilde{H}_{\ell_1+1} \cdots \times \tilde{H}_{\ell}\Big) \setminus Z$.

\noindent Since $(B-B) \cap \Big(I(h^{d_1}_1)+\cdots+I(h^{d_{\ell}}_{\ell})\Big)\subseteq \{0\}$, we see that there are no solutions to \begin{equation*}a-b\equiv h(\n) \mod{L}, \quad a\in B, \  b \in B_1, \  \n\in H. \end{equation*}  
Combined with (\ref{symdif2}) and the orthogonality of the characters, this implies
\begin{equation*} \sum_{t \in \Z_L} \widehat{B}(t)\bar{\widehat{B_1}(s+t)}T(t) =\frac{1}{w^{\ell_2+\ell_3}L^{\ell+1}}\sum_{\substack{x\in\Z_L \\ \n \in H}} h'(\n)B(x+h(\n))B_1(x)e^{2\pi i xs/L}= O_{\textbf{h}}((\min wM_i)^{-1}),\end{equation*} 
where $$h'(\n)=\prod_{i=1}^{\ell}(h_i^{d_i})'(n_i),$$ $$T_i(t)=\frac{1}{L}\sum_{n=1}^{M_i} (h_i^{d_i})'(n)e^{2\pi i h_i^{d_i}(n) t/L}, \quad \tilde{T}_i(t)=\frac{1}{wL}\sum_{\substack{n=1 \\ (n,W)=1}}^{M_i} (h_i^{d_i})'(n)e^{2\pi i h_i^{d_i}(n) t/L},$$ and $$T(t)=\prod_{i=1}^{\ell_1}T_i(t)\prod_{j=\ell_1+1}^{\ell}\tilde{T}_j(t),$$ which immediately yields  
\begin{align*}\sum_{t\in \Z_{L} \setminus \{0\}} |\widehat{B}(t)||\widehat{B_1}(s+t)||T(t)|& \geq \Big|\sum_{t\in \Z_L \setminus \{0\}} \widehat{B}(t)\bar{\widehat{B_1}(s+t)}T(t)\Big| \\ &= \widehat{B}(0)|\widehat{B_1}(s)|T(0) - O_\textbf{h}((\min wM_i)^{-1}).\end{align*} 
Therefore, since $|\widehat{B_1}(s)| \geq \sigma/U$, $T(0)\geq 1/(4\ell)^{\ell}$, and $\sigma^{-1}, U\leq \cQ$, we have that  
\begin{equation}\label{massW}\sum_{t\in \Z_L \setminus \{0\}} |\widehat{B}(t)||\widehat{B_1}(s+t)||T(t)| \geq \frac{\sigma^2}{(5\ell)^{\ell}U}.
\end{equation}  It follows from traditional Weyl sum estimates and Lemmas 11 and 28 of \cite{Lucier} that 
\begin{equation}\label{Wmin} |T(t)| \leq  \frac{\sigma}{(10\ell)^{\ell}U} \quad \text{for all } t \in \mathfrak{m}(\eta^{-1},\eta^{-(D+\epsilon)}), \end{equation} 
provided we choose $c_0$ sufficiently small, and \begin{equation} \label{Wmaj} |T(t)| \ll_{\textbf{h},\epsilon} C^{\omega(q)}q^{-1/D}\min\{1, (L|t/L-a/q|)^{-1} \} \end{equation} 
if $t \in \mathbf{M}_{a/q}(\eta^{-1}),$ $(a,q)=1$, and $q\leq \eta^{-(D+\epsilon)}$, where $\omega(q)$ is the number of distinct prime factors of $q$ and $C=C(\textbf{k})$. Estimates (\ref{Wmin}) and (\ref{Wmaj}) follow from the same ingredients as (\ref{SmajII}) and (\ref{SminII}), and are discussed further in Appendix \ref{apxA}.

\noindent We have by (\ref{Wmin}), Cauchy-Schwarz, and Plancherel's Identity that 
\begin{equation*} \sum_{t \in \mathfrak{m}(\eta^{-1},\eta^{-(D+\epsilon)}) } |\widehat{B}(t)||\widehat{B_1}(t)||T(t)|  \leq \frac{\sigma^2}{(10\ell)^{\ell}U},  \end{equation*} 
which together with (\ref{mass}) yields 
\begin{equation} \label{Mmass} \sum_{t \in \mathfrak{M}(\eta^{-1},\eta^{-(D+\epsilon)})  } |\widehat{B}(t)||\widehat{B_1}(t)||T(t)|  \geq \frac{\sigma^2}{(10\ell)^{\ell}U}.  \end{equation} 
We now wish to assert that we can ignore those frequencies in the major arcs at which the transform of $B$ or $B_1$ is particularly small. In order to make this precise, we first need to invoke a weighted version of known estimates on the higher moments of Weyl sums. Specifically, it follows from Theorem 1.1 of \cite{Wooley} that  
\begin{equation} \label{moment} \sum_{t \in \Z_L} |T(t)|^{m} \leq  C, \end{equation} 
where $C=C(m)$. Choosing a constant $0<c_1<((40\ell)^{\ell}C^{1/m})^{-m/2}$, where $C$ comes from (\ref{moment}), we define 
\begin{equation} \label{XY} X = \left\{ t\in \mathfrak{M}(\eta^{-1},\eta^{-(D+\epsilon)})  :\min\Big\{|\widehat{B}(t)|,|\widehat{B_1}(s+t)|\Big\} \leq c_1\sigma^{(m+1)/2}U^{-m/2}\right\}\end{equation}
and \begin{equation*}Y=\mathfrak{M}(\eta^{-1},\eta^{-(D+\epsilon)})\setminus X. \end{equation*} 
Using H\"{o}lder's Inequality to exploit the higher moment estimate on $T$, followed by Plancherel's Identity, we see that 
\begin{align*}\sum_{t \in X}|\widehat{B}(t)||\widehat{B_1}(s+t)||T(t)| &\leq \Big(\sum_{t \in X} |\widehat{B}(t)|^{\frac{m}{m-1}}|\widehat{B_1}(s+t)|^{\frac{m}{m-1}} \Big)^{\frac{m-1}{m}} \Big( \sum_{t \in \Z_L} |T(t)|^m \Big)^{\frac{1}{m}}\\ &\leq \frac{c_1^{2/m}\sigma^{\frac{m+1}{m}}}{U}\Big(\sum_{t \in \Z_L} \max\left\{|\widehat{B}(t)|^2,|\widehat{B_1}(s+t)|^2\right\}\Big)^{\frac{m-1}{m}}\cdot C^{1/m} \\ & \leq \frac{\sigma^{\frac{m+1}{m}}}{(40\ell)^{\ell}U}\Big(\sum_{t \in \Z_L} |\widehat{B}(t)|^2+|\widehat{B_1}(s+t)|^2\Big)^{\frac{m-1}{m}} \\ & \leq \frac{\sigma^2}{(20\ell)^{\ell}U},\end{align*}
and hence by (\ref{Mmass}) we have 
\begin{equation}\label{Ymass}  \sum_{t \in Y } |\widehat{B}(t)||\widehat{B_1}(s+t)||S(t)|  \geq \frac{\sigma^2}{(20\ell)^{\ell}U}. \end{equation} 
For $i,j,k \in \N$, we define 
\begin{equation*} \mathcal{R}_{i,j,k} = \left\{ a/q: (a,q)=1, \ 2^{i-1} \leq q \leq 2^i,  \text{ } \frac{\sigma}{2^j} \leq \max |\widehat{B}(t)| \leq \frac{\sigma}{2^{j-1}},\text{ } \frac{\sigma}{2^k} \leq \max |\widehat{B_1}(s+t)| \leq \frac{\sigma}{2^{k-1}} \right\}, \end{equation*}
where the maximums are taken over nonzero frequencies  $t\in \mathbf{M}_{a/q}(\eta^{-1})$. We see that we have 
\begin{equation}\label{rij}\sum_{a/q\in \mathcal{R}_{i,j,k}}\sum_{t\in \mathbf{M}_{a/q}(\eta^{-1})\setminus\{0\}} |\widehat{B}(t)||\widehat {B_1}(s+t)||T(t)| \ll |\mathcal{R}_{i,j,k}|\frac{\sigma^2}{2^{j}2^k}\max_{a/q \in \mathcal{R}_{i,j,k}}\sum_{t\in \mathbf{M}_{a/q}(\eta^{-1})} |T(t)|. \end{equation}
It follows from (\ref{Wmaj}), the bound $U,\sigma^{-1}\leq \cQ^{1/m}$, and the standard estimates $$\omega(q)\ll \log q/\log\log q, \quad q/\phi(q) \ll \log\log q,$$ that if $(a,q)=1$ and $q\leq \eta^{-(D+\epsilon)}$, then
\begin{align*}\sum_{t\in \mathbf{M}_{a/q}(\eta^{-1})} |T(t)| &\ll_{\textbf{k},\epsilon} C^{\omega(q)}(q/\phi(q))^{\ell_2+\ell_3}q^{-1/D}\log(\cQ) \\ & \ll_{\epsilon} q^{-1/D}(\log N)^{\epsilon},\end{align*} 
hence by (\ref{rij}), and the fact that $D\leq 2$, we have 
\begin{equation} \label{rijs}\sum_{a/q\in \mathcal{R}_{i,j,k}}\sum_{t\in \mathbf{M}_{a/q}(\eta^{-1})\setminus\{0\}} |\widehat{B}(t)||\widehat{B_1}(s+t)||T(t)| \ll_{\epsilon} |\mathcal{R}_{i,j,k}|\frac{\sigma^2}{2^{j}2^k}2^{-i/2}(\log N)^{\epsilon}. \end{equation}   

\noindent By our definitions, the sets $\mathcal{R}_{i,j,k}$ exhaust $Y$ by taking $1\leq 2^i \leq \eta^{-(D+\epsilon)}$ and $1 \leq 2^j,2^k \leq U^{m/2}/c_1\sigma^{(m-1)/2}$, a total search space of size $\ll_{\textbf{k},\epsilon} (\log \cQ)^3 $. Therefore, by (\ref{Ymass}) and (\ref{rijs}) there exist $i,j,k$ in the above range such that 
\begin{equation*}\frac{\sigma^2}{U(\log \cQ)^3} \ll_{\textbf{k},\epsilon}|\mathcal{R}_{i,j,k}|\frac{\sigma^2}{2^j2^k}2^{-i/2}(\log N)^{\epsilon}. \end{equation*}
In other words, we can set $V_s=2^{i}$, $W_s=2^j$, and $U_s=2^k$ and take an appropriate nonzero frequency from each of the pairwise disjoint major arcs specified by $\mathcal{R}_{i,j,k}$ to form a set 
\begin{equation*} P_s \subseteq \left\{t \in \bigcup_{q=V_s/2}^{V_s} \mathbf{M}_{q}(\eta^{-1}):  \text{ } |\widehat{B_1}(s+t)| \geq \frac{\sigma}{U_s} \right\} \end{equation*} 
which satisfies 
\begin{equation}\label{Ps} |P_s| \gg_{\epsilon} \frac{U_sW_sV_s^{1/2}}{U(\log N)^{2\epsilon}},  \quad |P_s\cap\mathbf{M}_{a,q}(\eta^{-1})|\leq 1 \text{ \ whenever \ } q \leq V_s,  \end{equation} and 
\begin{equation}\label{maxBW} \max_{ t \in \mathbf{M}_{a/q}(\eta^{-1})\setminus \{0\}} |\widehat{B}(t)| \geq \frac{\sigma}{W_s} \text{ whenever } q\leq V_s \text{ and }  \mathbf{M}_{a/q}(\eta^{-1})\cap P_s \neq \emptyset,
\end{equation} 

\noindent noting by disjointness that $a/q \in \mathcal{R}_{i,j,k}$  whenever $q\leq V_s$ and  $\mathbf{M}_{a/q}(\eta^{-1})\cap P_s \neq \emptyset $. 

\noindent We now observe that by the pigeonholing there is a subset $\tilde{P} \subseteq P$ with $|\tilde{P}| \gg_{\textbf{k},\epsilon} |P|/(\log \cQ)^3$, and hence
\begin{equation}\label{tilde}|\tilde{P}| \gg_{\epsilon} |P|/(\log N)^{\epsilon},\end{equation}  
for which the triple $U_s,W_s,V_s$ is the same. We call  those common parameters $\tilde{U},\tilde{W}$ and $\tilde{V}$, respectively, and we can now foreshadow by asserting that the claimed parameters in the conclusion of Lemma \ref{itn} will be $U'=\tilde{U}$, $V'=\tilde{V}V$, and $K'=K+\eta^{-1}$, which do satisfy the purported bound. 

\noindent We let 
\begin{equation*} \mathcal{R} = \left\{\frac{a}{q}+\frac{b}{r}:  s\in \mathbf{M}_{a/q}(K) \text{ for some } s \in \tilde{P} \text{ and }  t \in \mathbf{M}_{b/r}(\eta^{-1}) \text{ for some } t \in P_s \right\}. \end{equation*}

\noindent By taking one frequency $s+t$ associated to each element in $\mathcal{R}$,  we form our set $P'$, which immediately satisfies conditions (\ref{P'1}) and (\ref{P'2}) from the conclusion of Lemma \ref{itn}. However, the crucial condition (\ref{P'3}) on $|P'|$, which by construction is equal to $|\mathcal{R}|$, remains to be shown. To this end, we invoke the work on the combinatorics of rational numbers found in \cite{PSS} and \cite{BPPS}.

\begin{lemma}[Lemma CR of \cite{BPPS}]\label{CR} \begin{equation*} |\mathcal{R}|\geq \frac{|\tilde{P}|(\min_{s\in \tilde{P}}|P_s|)^2}{\tilde{V}E\tau^8(1+\log V)},  \end{equation*} where \begin{equation*} E= \max_{r\leq \tilde{V}} \Bigl|\Bigl\{ b  : \ (b,r)=1, \  \mathbf{M}_{b/r}(\eta^{-1})\cap \bigcup_{s\in \tilde{P}} P_s\neq \emptyset \Bigr\}\Bigr|, \end{equation*} $\tau(q)$ is the divisor function and $\tau = \max_{q\leq V\tilde{V}} \tau(q)$.
\end{lemma}

\noindent It is a well-known fact of the divisor function that $\tau(n) \leq n^{1/\log\log n}$ for large $n$, and since $\eta^{-1}, V\tilde{V} \leq \cQ$, we have that $\tau \leq (\log N)^{\epsilon}$. 

\noindent 
We also have from (\ref{maxmass}) that 
\begin{equation*} \sigma^2(\log N)^{-1+\epsilon} \geq \max_{r \leq \cQ} \sum_{t \in \mathbf{M}_r(\cQ)} |\widehat{B}(t)|^2 \geq \max_{r \leq \tilde{V}} \sum_{t \in \mathbf{M}_r(\eta^{-1})} |\widehat{B}(t)|^2 \geq \frac{\sigma^2}{\tilde{W}^2}E,\end{equation*}
where the last inequality follows from (\ref{maxBW}),  and hence 
\begin{equation}\label{D} E \leq \tilde{W}^2(\log N)^{-1 + \epsilon}.\end{equation}

 \noindent Combining the estimates on $\tau$ and $E$ with (\ref{Ps}), (\ref{tilde}), and Lemma \ref{CR}, we have\begin{equation*}|P'| \gg_{\epsilon} \frac{|P|}{(\log N)^{\epsilon}} \frac{\tilde{U}^2\tilde{W}^2\tilde{V}}{U^2(\log N)^{4\epsilon}} \frac{(\log N)^{1-\epsilon}}{\tilde{V}\tilde{W}^2(\log N)^{9\epsilon}} = \tilde{U}^2\frac{|P|}{U^2}(\log N)^{1-15\epsilon}.\end{equation*} \\
Recalling that we set $U'=\tilde{U}$, the lemma follows.
\qed

\noindent \textit{Remark on Theorem \ref{main2} for $1<D'\leq 2$.} To avoid excessive redundancy, we omit the details of the double iteration method with prime inputs, which establishes the bounds in Theorem \ref{main2} for $1<D'\leq2$. All of the tools required to adapt the argument from unrestricted inputs to prime inputs are already on display in Section \ref{primesec}'s adaptation of S\'ark\"ozy's method. For a detailed treatment of the double iteration method with prime inputs in the single polynomial case, the interested reader may refer to Chapter 11 of \cite{thesis}. 
\appendix

\section{Exponential Sum Estimates} \label{apxA}

\noindent In this appendix, we either invoke or prove all exponential sum estimates necessary to establish the crucial major and minor arc upper bounds in Sections \ref{unr} and \ref{Sec3}, namely (\ref{SmajII}), (\ref{SminII}), (\ref{Wmin}), and (\ref{Wmaj}). The first two lemmas provide asymptotic formulae for the relevant Weyl sums near rationals with small denominator.

\begin{lemma}\label{SasymI} Suppose $g(x)=a_0+a_1x+\cdots+a_jx^j \in \Z[x]$, $X\geq1$, and let $J=|a_0|+\cdots+|a_j|$. If $a,q\in \N$ and $\alpha=a/q+\beta$, then $$\sum_{n=1}^Xe^{2\pi \textnormal{i} g(n)\alpha}=q^{-1}\Big(\sum_{s=0}^{q-1} e^{2\pi \textnormal{i}g(s)a/q}\Big)\int_0^Xe^{2\pi \textnormal{i}g(x)\beta}\textnormal{d}x + O(q(1+JX^j\beta)).$$
\end{lemma}

\begin{proof} We begin by noting that for any $a,q \in \N$ and $x\geq0$,

\begin{equation}\label{rat2} \sum_{n=1}^x e^{2\pi \i g(n)a/q}=\sum_{s=0}^{q-1}\sum_{\substack{n=1 \\ n\equiv s \text{ mod }q}}e^{2\pi \i g(s)a/q}= q^{-1}\Big(\sum_{s=0}^{q-1}e^{2\pi \i g(s)a/q}\Big)x+ O(q),\end{equation}
\

\noindent since \begin{equation} \label{count} \# \{1\leq n \leq x : n\equiv s \text{ mod }q\}=x/q + O(1).\end{equation} 

\noindent Using (\ref{rat2}) and successive applications of summation and integration by parts, we have that if $\alpha=a/q+\beta$, then
\begin{align*}\sum_{n=1}^Xe^{2\pi \i g(n)\alpha} &= q^{-1}\Big(\sum_{s=0}^{q-1}e^{2\pi \i g(s)a/q}\Big)\Big(Xe^{2\pi \i g(X)\beta}-\int_0^Xx(2\pi \i \beta g'(x))e^{2\pi \i g(x)\beta}\textnormal{d}x\Big) + O(q(1+JX^{j}\beta))\\\\& =q^{-1}\Big(\sum_{s=0}^{q-1}e^{2\pi \i g(s)a/q}\Big)\int_0^Xe^{2\pi \textnormal{i} g(x)\beta}\textnormal{d}x + O(q(1+JX^j\beta)),
\end{align*} as required. 
\end{proof} 

\begin{lemma}\label{SasymW} Suppose $g(x)=a_0+a_1x+\cdots+a_jx^j \in \Z[x]$, $X\geq1$, $W=\prod_{p\leq Y} p$, and let $J=|a_0|+\cdots+|a_j|$. If $a,q\in \N$ and $\alpha=a/q+\beta$, then $$\sum_{\substack{n=1 \\ (n,W)=1}}^Xe^{2\pi \textnormal{i}g(n)\alpha}=q^{-1} \prod_{\substack{p\leq Y \\ p\nmid q}}\Big(1-\frac{1}{p}
\Big)\Big(\sum_{\substack{s=0 \\ \big((s,q),W\big)=1}}^{q-1}e^{2\pi \textnormal{i} g(s)a/q}\Big)\int_0^Xe^{2\pi\textnormal{i}g(x)\beta}\textnormal{d}x + O(Xe^{-\frac{\log(X/q)}{2\log Y}}(1+JX^j\beta)).$$
\end{lemma}

\begin{proof} Lemma \ref{SasymW} follows by simply mimicking the proof of Lemma \ref{SasymI}, replacing (\ref{count}) with the fact that if $((s,q),W)=1$, then $$\#\{1\leq n \leq x : n\equiv s \text{ mod } q, \ (n,W)=1\}= \frac{x}{q}\prod_{\substack{p\leq Y \\ p\nmid q}}\Big(1-\frac{1}{p}
\Big)\Big(1+O(e^{-\frac{\log(x/q)}{2\log Y}})\Big),$$ which follows from Theorem 7.2 in \cite{HR} as exhibited in \cite{BPPS}, whereas otherwise this set is empty.
\end{proof}

\noindent As indicated by the asymptotic formulae in Lemmas \ref{SasymI} and \ref{SasymW}, we can beat the trivial bound on the Weyl sums near rationals with small denominator by invoking nontrivial estimates on the Gauss sums, which are provided by the next three lemmas.

\begin{lemma}[Lemma 6, \cite{Lucier}]  \label{gauss2}If $g(x)=a_0+a_1x+\cdots+a_jx^j \in \Z[x]$, $j\geq 1$,  then 
$$\Big|\sum_{s=0}^{q-1}e^{2\pi \textnormal{i} g(s)/q}\Big|\ll_j \gcd(\textnormal{cont}(g),q)^{1/j}q^{1-1/j},$$ where $$\textnormal{cont}(g)=\gcd(a_1,\dots,a_j).$$
\end{lemma} 

\begin{lemma}[Lemma 2, \cite{BPPS}] \label{gaussB} If $(a,q)=1$, then $$\Big|\sum_{\substack{s=0 \\ (s,q)=1}}^{q-1}e^{2\pi \textnormal{i} s^ja/q}\Big| \ll C^{\omega(q)} q^{1/2},$$ where $C=C(j)$.
\end{lemma}

\begin{lemma}[Theorem 1, \cite{Shpar}] \label{gaussS} If $g \in \Z[x]$ satisfies $g(0)=0$ and has $r\geq 2$ nonzero coefficients, then 
$$\Big|\sum_{\substack{s=0 \\ (s,q)=1}}^{q-1}e^{2\pi \textnormal{i} g(s)/q}\Big|\leq C^{\omega(q)} \gcd(\textnormal{cont}(g),q)^{1/r}q^{1-1/r}.$$ 
\end{lemma} 
 
\noindent Lemmas \ref{gauss2} and \ref{gaussS} indicate that we could potentially lose control of the Gauss sums if the coefficients of the auxiliary polynomials gain larger and larger common factors during the iterations, but the following observation of Lucier asserts that this is not case.

\begin{lemma}[Lemma 28, \cite{Lucier}] \label{content} If $g\in \Z[x]$ is an intersective polynomial of degree $j$ and $g^d$ are the auxiliary polynomials as defined in Section \ref{auxsec}, then for every $d\in \N$, \begin{equation*} \textnormal{cont}(g^d) \leq |\Delta(g)|^{(j-1)/2}\textnormal{cont}(g), \end{equation*} where $\Delta(g)=a^{2j-2}\prod_{i\neq i'} (\alpha_i-\alpha_{i'})^{e_ie_{i'}}$ if $g$ factors over the complex numbers as $$g(x)=a(x-\alpha_1)^{e_1}\dots(x-\alpha_r)^{e_r}$$ with all the $\alpha_i$'s distinct.
\end{lemma}

\noindent The following standard result combines with Lemmas \ref{SasymI} and \ref{SasymW} to allow us to gain additional savings from the trivial bound on the Weyl sums close, but not too close, to rationals with very small denominator.

\begin{lemma}[Van der Corput's Lemma] \label{osc}If $j\geq 1$ and $X>0$, then  $$\Big|\int_0^Xe^{2\pi \textnormal{i} x^j\beta}\textnormal{d}x\Big| \leq 2|\beta|^{-1/j}.$$ 
\end{lemma}

\noindent Finally, we invoke a variation of the most traditional minor arc estimate, Weyl's Inequality, to get the desired estimates far from rationals with remotely small denominator.

\begin{lemma}[Lemma 3 in \cite{CLR}] \label{weyl2}  Suppose $g(x)=a_0+a_1x+\cdots+a_{j}x^{j}$ with $a_i \in \R$ and $a_{j} \in \N$. If $(a,q)=1$ and $|\alpha-a/q|<q^{-2}$, then 
$$ \Big|\sum_{n=1}^X e^{2\pi i g(n)\alpha} \Big| \ll_{j} X \Big(a_{j}\log^{j^2}(a_{j}qX)(1/q+1/X+q/a_{j}X^{j}) \Big)^{2^{-j}}.$$
\end{lemma}

\subsection{Proof of (\ref{SmajII}) and (\ref{SminII})} \label{A1} We return to the setting of the proof of Lemma \ref{L2I}, recalling all assumptions,  notation, and fixed parameters. Fixing $\alpha \in \T$ and letting $Z=N^{2^{-5k}}$, the pigeonhole principle guarantees the existence of $1\leq q \leq N/Z$ and $(a,q)=1$ with $$\Big|\alpha-\frac{a}{q}\Big|<\frac{Z}{qN}.$$ Letting $\beta=\alpha-a/q$, we have by Lemmas \ref{SasymI} and \ref{SasymW} that

\begin{equation}\label{blah1} S_i(\alpha)=q^{-1}\Big(\sum_{s=0}^{q-1}e^{2\pi \textnormal{i} h^{d_i}_i(s)a/q}\Big) \int_0^{M_i}e^{2\pi \textnormal{i} h^{d_i}_i(x) \beta}\textnormal{d}x + O_{h_i}(q(1+N\beta))
\end{equation} 
and 
\begin{equation}\label{blah2} \tilde{S}_i(\alpha)=\frac{w_q}{q}\Big(\sum_{\substack{s=0 \\ \big((s,q),W\big)=1}}^{q-1}e^{2\pi \textnormal{i} h^{d_i}_i(s)a/q}\Big) \int_0^{M_i}e^{2\pi \textnormal{i} h^{d_i}_i(x) \beta}\textnormal{d}x + O_{h_i}(M_ie^{-\frac{\log(M_i/q)}{4D'\log(\delta^{-1})}}(1+N\beta))
\end{equation}  
for $1\leq i \leq \ell$, where $$w_q= \prod_{\substack{p\leq \eta^{-(D'+\epsilon)}\\ p\nmid q}}\Big(1-\frac{1}{p}
\Big).$$ In particular, if $q\leq \eta^{-(D'+\epsilon)}$, then 
\begin{equation*}S_i(\alpha)=q^{-1}\Big(\sum_{s=0}^{q-1}e^{2\pi \textnormal{i} h^{d_i}_i(s)a/q}\Big) \int_0^{M_i}e^{2\pi \textnormal{i} h^{d_i}_i(x) \beta}\textnormal{d}x + O_{h_i}(M_i\cQ^{-K})
\end{equation*} 
and 
\begin{equation*} \tilde{S}_i(\alpha)=\frac{\tilde{M_i}}{\phi(q)M_i}\Big(\sum_{\substack{s=0 \\ (s,q)=1}}^{q-1}e^{2\pi \textnormal{i} h^{d_i}_i(s)a/q}\Big) \int_0^{M_i}e^{2\pi \textnormal{i} h^{d_i}_i(x) \beta}\textnormal{d}x + O_{h_i}(M_i\cQ^{-K}).
\end{equation*}
If $q\leq Q$ and $|\beta|<\gamma$, then (\ref{SmajII}) follows from the definition of $S(\alpha)$ by trivially bounding the above integrals and applying Lemmas \ref{gauss2}, \ref{gaussB}, \ref{gaussS}, and \ref{content}.

\noindent If $q\leq Q$ and $\gamma\leq |\beta| < \cQ^K/qN$, then after trivially bounding the exponential sums in (\ref{blah1}) and (\ref{blah2}), (\ref{SminII}) follows from Lemma \ref{osc} and the observation that \begin{equation} \label{osc3} \Big|\int_0^Me^{2\pi \textnormal{i} b_i^{d_i}x^{k_i}\beta}-e^{2\pi \textnormal{i}h_i^{d_i}(x)\beta} \textnormal{d}x\Big| \ll_{h_i} (d_iM_i)^{k_i}\beta \leq M_i^{1/2},\end{equation} where $b_i^{d_i}$ is the leading coefficient of $h_i^{d_i}$.

\noindent For $Q<q\leq \eta^{-(D'+\epsilon)}$, we have shown that (\ref{SmajII}) holds, and in this range (\ref{SminII}) follows from (\ref{SmajII}).  

\noindent If $\eta^{-(D'+\epsilon)}<q\leq Z$, the exponential sum in (\ref{blah2}) does not collapse quite as conveniently, but by appropriately separating the sum and applying Lemmas \ref{gauss2} and \ref{content} we still have \begin{equation} \label{blah3} \Big|\sum_{\substack{s=0 \\ \big((s,q),W\big)=1}}^{q-1}e^{2\pi \textnormal{i} h^{d_i}_i(s)a/q}\Big| \ll_{h_i,\ell,\epsilon} q^{1-1/k_i+\epsilon/2\ell}.\end{equation} The deduction of this estimate is a simpler version of the proof of Lemma \ref{gsPI}, and (\ref{SminII}) then follows from trivially bounding the integrals in (\ref{blah1}) and (\ref{blah2}) and applying (\ref{blah3}) and Lemmas \ref{gauss2} and \ref{content}.

\noindent Finally, if $Z < q \leq N/Z$, then (\ref{SminII}) follows with room to spare from Lemma \ref{weyl2}, as exhibited in Section 4 of \cite{BPPS}, and the desired estimates are established in all cases. \qed

\subsection{Discussion of (\ref{Wmin}) and (\ref{Wmaj})} The weighted exponential sum estimates (\ref{Wmin}) and (\ref{Wmaj}) are obtained by mimicking Section \ref{A1}, applying Lemmas \ref{gauss2}-\ref{content}, as well as a weighted version of Lemma \ref{SasymI}, which follows analogously, and a weighted version of Lemma \ref{weyl2}, which follows from summation by parts. To highlight the reasoning for applying the derivative weight, we note that    $$\Big|\int_0^{M_i}(h^{d_i}_i)'(x)e^{2\pi \textnormal{i} h^{d_i}_i(x) \beta}\textnormal{d}x\Big| =\Big|\int_0^{h^{d_i}_i(M_i)}e^{2\pi \textnormal{i} y \beta}\textnormal{d}y\Big|\ll \min\{L,|\beta|^{-1}\}, $$ yielding the last term on the right hand side of (\ref{Wmaj}) and the tolerable logarithmic accumulation in summing over an entire major arc, which is crucial for the rest of the argument.
   
\section{Exponential Sum Estimates over Shifted Primes} \label{apxB} 

\noindent In this appendix, we prove or invoke prime input analogs of the estimates in Appendix \ref{apxA}, which are required to establish (\ref{WmajPI}) and (\ref{WminPI}). We begin with the following analog of Lemma \ref{SasymI}.

\begin{lemma}\label{WasymPI} Let $N, \cQ, k, K, q_o,$ and $\chi$ be as fixed in Section \ref{primesec}.   Suppose $g(x)=a_0+a_1x+\cdots+a_jx^j\in \Z[x]$ and let $J=|a_0|+\cdots+|a_j|$. If $q_0 \mid d$, $d/q_0\leq \cQ$, $a,q\in \N$, $q\leq (q_0\cQ)^{8K},$ $X\geq N^{1/10k}$, and $\alpha=a/q+\beta$, then
\begin{align*} \sum_{\substack{n=1 \\ dn+r \text{prime}}}^X \log(dn+r)e^{2\pi \textnormal{i} g(n)\alpha} &= \frac{d}{\phi(qd)} \mathcal{G}(a,q)\int_0^X(1-\chi(r)(dx)^{\rho-1})e^{2\pi \textnormal{i} g(x)\beta}\textnormal{d}x+ O(qX(1+JX^j\beta)\cQ^{-900K^2}),
\end{align*}  where $$\mathcal{G}(a,q)=\sum_{\substack{s=0 \\ (ds+r,q)=1}}^{q-1} e^{2\pi \textnormal{i} g(s)a/q}.$$
\end{lemma}
\begin{proof} Lemma \ref{WasymPI} follows by mimicking the proof of Lemma \ref{SasymI}, replacing (\ref{count}) with  (\ref{PsiAsymPI}) and the observation that $$ \sum_{\substack{n=1 \\ n \equiv s \text{ mod }q \\ dn+r \text{ prime}}}^X \log(dn+r)=\psi(dX+r,ds+r,qd).$$\end{proof} 

\noindent In place of Lemmas \ref{gauss2}-\ref{gaussS}, we apply the following restricted Gauss sum estimate.

\begin{lemma}\label{gsPI} Suppose $g(x)=a_0+a_1x+\cdots+a_jx^j \in \Z[x]$. If $d,r\in \Z$, $q\in \N$ and $(a,q)=1$, then
\begin{equation} \Big|\sum_{\substack{s=0 \\ (ds+r,q)=1}}^{q-1} e^{2\pi \textnormal{i} g(s) a/q}\Big| \leq C^{\omega(q)} \Big(\gcd(\textnormal{cont}(g),q_1)\gcd(a_j,q_2)\Big)^{1/j} q^{1-1/j},
\end{equation} 
where $C=C(j)$, $q=q_1q_2$, and $q_2$ is the maximal divisor of $q$ which is coprime to $d$.
\end{lemma}

\begin{proof}
Fix $g, d, r,a,q$ as in Lemma \ref{gsPI}. As is often the case with this type of sum, we can simplify our argument by taking advantage of multiplicativity. Specifically, it is not difficult to show that if  $q=q_1q_2$ with $(q_1,q_2)=1$, then  
\begin{equation*}\sum_{\substack{s=0 \\ (ds+r,q)=1}}^{q-1} e^{2\pi \text{i} g(s) a/q}= \Bigg( \sum_{\substack{s_1=0 \\ (ds_1+r,q_1)=1}}^{q_1-1} e^{2\pi \text{i} g(s_1) a_1/q_1} \Bigg) \Bigg(\sum_{\substack{s_2=0 \\ (ds_2+r,q_2)=1}}^{q_2-1} e^{2\pi \text{i} g(s_2) a_2/q_2}\Bigg),
\end{equation*} where $a/q=a_1/q_1+a_2/q_2$, so we can assume $q=p^{v}$ for some $p\in \P$, $v \in \N$. If $p \mid d$ and $p \mid r$, then $ds+r$ is never coprime to $p^{v}$, so the sum is clearly zero. If $p \mid d$ and $p \nmid r$, then $ds+r$ is always coprime to $p^{v}$, so the sum is complete and the result follows from Lemma \ref{gauss2}. If $p \nmid d$, then $p \mid ds+r$ if and only if $s\equiv -rd^{-1}$ mod $p$. Therefore, 
\begin{equation} \sum_{\substack{s=0 \\ p\nmid ds+r}}^{p^{v}-1} e^{2\pi \text{i} g(s) a/p^{v}}= \sum_{s=0}^{p^{v}-1} e^{2\pi \text{i} g(s) a/p^{v}}- \sum_{s'=0}^{p^{v-1}-1} e^{2\pi \text{i} g(ps'+m) a / p^{v}},
\end{equation} where $m \equiv -rd^{-1}$ mod $p$, and by Lemma \ref{gauss2} we need only obtain the estimate for the second sum. Setting 
\begin{equation*} \tilde{g}(s)=\frac{g(ps+m)-g(m)}{p},
\end{equation*}
we see that $\tilde{g}$ is a polynomial with integer coefficients and leading coefficient $a_jp^{v-1}$. In particular, 
\begin{equation*}\gcd(\text{cont}(\tilde{g}),p^{v-1}) \leq p^{v-1}\gcd(a_j,p^{v-1}).\end{equation*} Therefore, by Lemma \ref{gauss2} we have
\begin{align*}\Big|\sum_{s=0}^{p^{v-1}-1} e^{2\pi \text{i} g(ps+m) a / p^{v}}\Big|&= \Big|\sum_{s=0}^{p^{v-1}-1}e^{2\pi \text{i} (f(ps+m)-f(m)) a / p^{v}}\Big| \\ &= \Big|\sum_{s=0}^{p^{v-1}-1}e^{2\pi \text{i}\tilde{g}(s)a/p^{v-1}} \Big| \\& \ll_j \Big(p^{v-1}\gcd(a_j,p^{v-1})\Big)^{1/j}p^{(v-1)(1-1/j)} \\&\leq \gcd(a_j,p^{v})^{1/j}p^{v(1-1/j)},
\end{align*}
as required.
\end{proof}

\noindent \textit{Remark.} The factor of $C^{\omega(q)}$ in the conclusion of Lemma \ref{gsPI} arises when exploiting multiplicativity after decomposing $q$ into a product of prime powers. In the published version of \cite{Rice}, this factor is incorrectly absent in the corresponding Lemma 9, while the identical proof is provided. Fortunately, this oversight has no bearing on the main results of that paper (a corrected version has been uploaded to the arxiv server), but the distinction is relevant in certain cases here.

\noindent In the case that $\ell=k_1=1$, that is to say the case of $p\pm 1$ differences or a fixed multiple thereof, we can evaluate the relevant local sum precisely, showing that it has magnitude at most $1$. Here we restrict to the $p-1$ case, and the $p+1$ case is analogous.

\begin{lemma} \label{mob} If $(a,q)=1$ and $d \in \Z$, then $$\sum_{\substack{s=0 \\ (ds+1,q)=1}}^{q-1}e^{2\pi \textnormal{i}sa/q}=\begin{cases} \mu(q)e^{-2\pi \textnormal{i}ja/q} &\text{if } (d,q)=1, \text{ where } j \equiv d^{-1} \textnormal{ mod } q \\ 0 &\text{else} \end{cases},$$
where $\mu$ is the M\"obius function.
\end{lemma}

\begin{proof}Again we exploit multiplicativity, and we see from the Chinese Remainder Theorem that if $(a,q)=1$ and $q=q_1q_2$ with $(q_1,q_2)=1$, then 
$$ \sum_{\substack{r=0 \\ (ds+1,q)=1}}^{q-1}e^{2\pi \textnormal{i}sa/q}=\sum_{\substack{r_1=0 \\ (ds_1+1,q_1)=1}}^{q_1-1}e^{2\pi \textnormal{i}s_1a_1/q_1}\cdot \sum_{\substack{s_2=0 \\ (ds_2+1,q_2)=1}}^{q_2-1}e^{2\pi \textnormal{i}s_2a_2/q_2},$$
where $a/q=a_1/q_1+a_2/q_2$. Therefore, we can assume $q=p^v$ for $p\in \P$ and $v\in \N$. If $p\mid d$, then we always have $(ds+1,p^v)=1$, so the exponential sum is complete and equal to $0$ by orthogonality.  

\noindent If $p\nmid d$, then we can change variables in the sum setting $r=ds+1$, which yields
$$\sum_{\substack{r=0 \\ p \nmid r}}^{p^v-1} e^{2\pi \textnormal{i} (r-1)j a/p^v},$$ where $j \equiv d^{-1}$ mod $p^v$, so the lemma follows from the identity 
$$\sum_{\substack{r=0 \\ p \nmid r}}^{p^v-1} e^{2\pi \textnormal{i} ra/p^v}=\begin{cases} -1 &\text{if }v=1 \\ 0 &\text{else} \end{cases},$$ which again follows from orthogonality. 
\end{proof}

\noindent In place of Lemma \ref{weyl2}, we invoke the following prime input analog, which is a less precise, only nominally generalized version of Theorem 4.1 in \cite{lipan}.

\begin{lemma}[Lemma 4 in \cite{CLR}] \label{PImin} Suppose $g(x)=a_0+a_1x+\cdots+a_{j}x^{j} \in \Z[x]$,  and let $J=64j^24^{j}$. If $U\geq \log X$,
$|a_{j}| \geq C\Big(|a_{\ell-1}| + \cdots +|a_0|\Big),$ and $|d|,|r|,|a_{j}| \leq U^{j}$, then 
\begin{equation*}\sum_{\substack{n=1 \\ dn+r \text{ prime}}}^X \log(dn+r)e^{2\pi \textnormal{i} g(n) \alpha} \ll_C \frac{X}{U}+U^JX^{1-4^{-j}}
\end{equation*} 
provided \begin{equation*} |\alpha -a/q| < q^{-2} \quad \text{for some} \quad U^{J} \leq q \leq h(X)/U^{J} \quad \text{and} \quad (a,q)=1.
\end{equation*}
\end{lemma}

\subsection{Proof of (\ref{WmajPI}) and (\ref{WminPI})} We return to the setting of the proof of Lemma \ref{PIL2}, recalling all assumptions,  notation, and fixed parameters. Here we establish (\ref{WmajPI}) and (\ref{WminPI}) by mimicking the proofs of (\ref{SmajII}) and (\ref{SminII}). The details can be fleshed out by referring to section \ref{A1}, and also applying (\ref{m/q3K}) when necessary.    
 
\noindent Fixing $\alpha \in \T$, the pigeonhole principle guarantees the existence of $1\leq q \leq L/(q_0\cQ)^{2K}$ and $(a,q)=1$ with $$\Big|\alpha-\frac{a}{q}\Big|<\frac{(q_0\cQ)^{2K}}{qL}.$$
Letting $\beta=\alpha-a/q$, we have that if $q\leq (q_0\cQ)^{2K}$, then (\ref{WmajPI}) follows from Lemmas \ref{WasymPI}-\ref{mob}. In particular, (\ref{WmajPI}) holds if $q\leq Q$ and $|\beta|<\gamma$. 

\noindent If $q\leq Q$ and $\gamma < |\beta| \leq \dfrac{(q_0\cQ)^{2K}}{qL}$, then (\ref{WminPI}) follows from Lemma \ref{WasymPI}, Lemma \ref{osc}, and integration by parts. If $Q < q \leq (q_0\cQ)^{2K}$, then, as previously mentioned, (\ref{WmajPI}) holds, and in this range (\ref{WmajPI}) implies (\ref{WminPI}). 

\noindent Finally, if $(q_0\cQ)^{2K}<q\leq L/(q_0\cQ)^{2K}$, then (\ref{WminPI}) follows from applying Lemma \ref{PImin} with $U=(q_0\cQ)^2$. \qed

\end{document}